\documentclass[english,12pt]{smfart}
\usepackage{amssymb}
\usepackage{amsfonts}
\usepackage{amscd}
\usepackage[all]{xypic}
\usepackage{amsmath}
\usepackage{stmaryrd}
\usepackage{mathrsfs}
\usepackage{color} 
\usepackage{url}
\usepackage{cancel}
  \newtheorem{lem}[subsubsection]{Lemma}
  \newtheorem{prop}[subsubsection]{Proposition}
   \newtheorem{theo}[subsubsection]{Theorem}
  \newtheorem{cor}[subsubsection]{Corollary}
    \newtheorem{Defi}[subsubsection]{Definition}

\usepackage{lmodern}
\usepackage[all]{xypic}

\theoremstyle{plain}
\newtheorem*{theo*}{\theoname}
\theoremstyle{remark}

\newtheorem{rema}[subsubsection]{Remark}

\def\a{\alpha}
\def\b{\beta}
\def\prf{\begin{proof}} 
\def\eprf{\end{proof}}
\def\tensor{{\otimes}}
\def\meet{\cap}
\def\union{\cup}

\def\g{\gamma}
\def\G{\Gamma}

\def\<{\begin}
 \def\>{\end}
 \def\tcb{}

\let\phi\varphi

\def\Spec{\operatorname{Spec}}

\def\wt{\operatorname{w}}

\newcommand{\nc}{\newcommand}
\nc{\renc}{\renewcommand}
\nc{\ssec}{\subsection}
\nc{\sssec}{\subsubsection}
\nc{\on}{\operatorname}

\nc\ol{\overline}

\nc\wh{\widehat}
\renc{\d}{{\delta}}
\nc{\Aa}{{\mathbb{A}}}
 \nc{\Gg}{{\mathbb{G}}}  
\nc{\Hh}{{\mathbb{H}}}
 \nc{\Nn}{{\mathbb{N}}}
\nc{\Pp}{{\mathbb{P}}}
\nc{\Rr}{{\mathbb{R}}}
\nc{\BV}{{\mathbb{V}}}
\nc{\BW}{{\mathbb{W}}}
\nc{\Zz}{{\mathbb{Z}}}
\nc{\Qq}{{\mathbb{Q}}}
\nc{\Ss}{{\mathbb{S}}}
\nc{\Cc}{{\mathbb{C}}}
\nc{\Ff}{{\mathbb{F}}}

\nc{\CA}{{\mathcal{A}}}
\nc{\CB}{{\mathcal{B}}}

\nc{\CE}{{\mathcal{E}}}
\nc{\CF}{{\mathcal{F}}}
\nc{\CG}{{\mathcal{G}}}
\nc{\CL}{{\mathcal{L}}}
\nc{\CC}{{\mathcal{C}}}
\nc{\CM}{{\mathcal{M}}}
\def\Mm{\CM}
\nc{\CN}{{\mathcal{N}}}
\nc{\Oo}{{\mathcal{O}}}
\nc{\CP}{{\mathcal{P}}}
\nc{\CQ}{{\mathcal{Q}}}
\nc{\CR}{{\mathcal{R}}}
\nc{\CS}{{\mathcal{S}}}
\nc{\CT}{{\mathcal{T}}}
\nc{\CU}{{\mathcal{U}}}
\nc{\CV}{{\mathcal{V}}}
\nc{\CK}{{\mathcal{K}}}
\nc{\CW}{{\mathcal{W}}}
\nc{\CZ}{{\mathcal{Z}}}

\nc{\cM}{{\check{\mathcal M}}{}}
\nc{\csM}{{\check{\mathcal A}}{}}
\nc{\oM}{{\overset{\circ}{\mathcal M}}{}}
\nc{\obM}{{\overset{\circ}{\mathbf M}}{}}
\nc{\oCA}{{\overset{\circ}{\mathcal A}}{}}
\nc{\obA}{{\overset{\circ}{\mathbf A}}{}}
\nc{\ooM}{{\overset{\circ}{M}}{}}
\nc{\osM}{{\overset{\circ}{\mathsf M}}{}}
\nc{\vM}{{\overset{\bullet}{\mathcal M}}{}}
\nc{\nM}{{\underset{\bullet}{\mathcal M}}{}}
\nc{\oD}{{\overset{\circ}{\mathcal D}}{}}
\nc{\obD}{{\overset{\circ}{\mathbf D}}{}}
\nc{\oA}{{\overset{\circ}{\mathbb A}}{}}
\nc{\op}{{\overset{\bullet}{\mathbf p}}{}}
\nc{\cp}{{\overset{\circ}{\mathbf p}}{}}
\nc{\oU}{{\overset{\bullet}{\mathcal U}}{}}
\nc{\oZ}{{\overset{\circ}{\mathcal Z}}{}}
\nc{\ofZ}{{\overset{\circ}{\mathfrak Z}}{}}
\nc{\oF}{{\overset{\circ}{\fF}}}

\nc{\fa}{{\mathfrak{a}}}
\nc{\fb}{{\mathfrak{b}}}
\nc{\fgl}{{\mathfrak{gl}}}
\nc{\fh}{{\mathfrak{h}}}
\nc{\fj}{{\mathfrak{j}}}
\nc{\fm}{{\mathfrak{m}}}
\nc{\fn}{{\mathfrak{n}}}
\nc{\fu}{{\mathfrak{u}}}
\nc{\fp}{{\mathfrak{p}}}
\nc{\fr}{{\mathfrak{r}}}
\nc{\fs}{{\mathfrak{s}}}
\nc{\fsl}{{\mathfrak{sl}}}
\nc{\hsl}{{\widehat{\mathfrak{sl}}}}
\nc{\hgl}{{\widehat{\mathfrak{gl}}}}
\nc{\hg}{{\widehat{\mathfrak{g}}}}
\nc{\chg}{{\widehat{\mathfrak{g}}}{}^\vee}
\nc{\hn}{{\widehat{\mathfrak{n}}}}
\nc{\chn}{{\widehat{\mathfrak{n}}}{}^\vee}

\nc{\fA}{{\mathfrak{A}}}
\nc{\fB}{{\mathfrak{B}}}
\nc{\fD}{{\mathfrak{D}}}
\nc{\fE}{{\mathfrak{E}}}
\nc{\fF}{{\mathfrak{F}}}
\nc{\fG}{{\mathfrak{G}}}
\nc{\fK}{{\mathfrak{K}}}
\nc{\fL}{{\mathfrak{L}}}
\nc{\fM}{{\mathfrak{M}}}
\nc{\fN}{{\mathfrak{N}}}
\nc{\fP}{{\mathfrak{P}}}
\nc{\fU}{{\mathfrak{U}}}
\nc{\fV}{{\mathfrak{V}}}
\nc{\fZ}{{\mathfrak{Z}}}

\nc{\bb}{{\mathbf{b}}}
\nc{\bc}{{\mathbf{c}}}
\nc{\bd}{{\mathbf{d}}}
\nc{\be}{{\mathbf{e}}}
\nc{\bj}{{\mathbf{j}}}
\nc{\bn}{{\mathbf{n}}}
\nc{\bp}{{\mathbf{p}}}
\nc{\bq}{{\mathbf{q}}}
\nc{\bF}{{\mathbf{F}}}
\nc{\bu}{{\mathbf{u}}}

\nc{\bv}{{\mathbf{v}}}
\nc{\bx}{{\mathbf{x}}}
\nc{\bs}{{\mathbf{s}}}
\nc{\by}{{\mathbf{y}}}
\nc{\bw}{{\mathbf{w}}}
\nc{\bA}{{\mathbf{A}}}
\nc{\bK}{{\mathbf{K}}}
\nc{\bI}{{\mathbf{I}}}
\nc{\bB}{{\mathbf{B}}}
\nc{\bG}{{\mathbf{G}}}
\nc{\bC}{{\mathbf{C}}}
\nc{\bD}{{\mathbf{D}}}
\nc{\bP}{{\mathbf{P}}}
\nc{\bH}{{\mathbf{H}}}
\nc{\bM}{{\mathbf{M}}}
\nc{\bN}{{\mathbf{N}}}
\nc{\bV}{{\mathbf{V}}}
\nc{\bU}{{\mathbf{U}}}
\nc{\bL}{{\mathbf{L}}}
\nc{\bT}{{\mathbf{T}}}
\nc{\bW}{{\mathbf{W}}}
\nc{\bX}{{\mathbf{X}}}
\nc{\bY}{{\mathbf{Y}}}
\nc{\bZ}{{\mathbf{Z}}}
\nc{\bS}{{\mathbf{S}}}

\nc{\sA}{{\mathsf{A}}}
\nc{\sB}{{\mathsf{B}}}
\nc{\sC}{{\mathsf{C}}}
\nc{\sD}{{\mathsf{D}}}
\nc{\sF}{{\mathsf{F}}}
\nc{\sG}{{\mathsf{G}}}
\nc{\sK}{{\mathsf{K}}}
\nc{\sM}{{\mathsf{M}}}
\nc{\sO}{{\mathsf{O}}}
\nc{\sQ}{{\mathsf{Q}}}
\nc{\sP}{{\mathsf{P}}}
\nc{\sZ}{{\mathsf{Z}}}
\nc{\sfp}{{\mathsf{p}}}
\nc{\sr}{{\mathsf{r}}}
\nc{\sg}{{\mathsf{g}}}
\nc{\sff}{{\mathsf{f}}}
\nc{\sfb}{{\mathsf{b}}}
\nc{\sfc}{{\mathsf{c}}}
\nc{\sd}{{\ltimes}}

\nc{\tA}{{\widetilde{\mathbf{A}}}}
\nc{\tB}{{\widetilde{\mathcal{B}}}}

\nc{\tG}{{\widetilde{G}}}
\nc{\TM}{{\widetilde{\mathbb{M}}}{}}
\nc{\tO}{{\widetilde{\mathsf{O}}}{}}
\nc{\tU}{\widetilde{U}}
\nc{\TZ}{{\tilde{Z}}}
\nc{\tx}{{\tilde{x}}}
\nc{\tq}{{\tilde{q}}}

\nc{\val}{{\rm val}}

\def\RV{\mathrm{RV}}
 \def\RES{{\mathrm{RES}}}
  \def\Var{{\mathrm{Var}}}
  \def\ACVF{{\mathrm {ACVF}}}
\def\VF{\mathrm{VF}}
\def\kk{\mathbf{k}}
\def\dcl{{\mathrm {dcl}}}

  \def\mC{\mathcal{C}}
    
\def\inv{^{-1}}
\def\rv{{\rm rv}}
\def\sp{{\rm sp}}
\def\tr{{\rm tr}}
\def\liminv{\underset{\longleftarrow}{\rm{lim}\,} }

\def\eu{\operatorname{eu}}
\def\EU{\operatorname{EU}}
\def\etEU{\EU_{\text{\'et}}}
\def\GEU{\EU_{\Gamma}}

\def\llp{\mathopen{(\!(}}
\def\llb{\mathopen{[\![}}
\def\rrp{\mathopen{)\!)}}
\def\rrb{\mathopen{]\!]}}  
\def\01{\left(0,1\right]}  
\numberwithin{equation}{subsection}
\begin{document}

\title{Monodromy and the Lefschetz fixed point formula}
\author{Ehud Hrushovski}

\address{Department of Mathematics, The Hebrew University, Jerusalem, Israel}
 \email{ehud@math.huji.ac.il}
\author{Fran\c cois Loeser}
\address{Institut de Math\'ematiques de Jussieu,
UMR 7586 du CNRS,
Universit\'e Pierre et Marie Curie,
Paris, France}
\email{Francois.Loeser@upmc.fr}
%


\dedicatory{To Jan  Denef as a token of admiration and friendship
\\on the occasion of his 60th birthday}

\maketitle

\section{Introduction}
\subsection{}Let $X$ be a smooth complex algebraic variety of dimension $d$
and let $f \colon X \rightarrow \mathbb{A}^1_{\mathbb{C}}$ be a non-constant morphism to the affine line.
Let $x$ be a singular point of  $f^{-1} ( 0)$,
that is, such that  $df (x) = 0$.

Fix a distance function $\delta$ on an open neighborhood  of $x$ induced from a local embedding of this neighborhood in some complex affine space.
For $\varepsilon >0$ small enough, one may consider the corresponding closed ball $B (x, \varepsilon)$ of radius $\varepsilon$ around $x$.
For $\eta >0$ we  denote by $D_{\eta}$ the closed disk of radius $\eta$ around the origin in $\mathbb{C}$.

By Milnor's local fibration Theorem (see \cite{Milnor}, \cite{Dimca}),
there exists $\varepsilon_0 >0$ such that, for every $0 < \varepsilon < \varepsilon_0$,
there exists $0 < \eta <\varepsilon$ such that
the morphism
$f$ restricts to  a fibration, called the Milnor fibration,
\begin{equation}
B (x, \varepsilon) \cap f^{-1}(D_{\eta} \setminus \{0\})
\longrightarrow
D_{\eta} \setminus \{0\}.
\end{equation}

The Milnor fiber at $x$,
\begin{equation}\label{mfdef}
F_x  = f^{-1} (\eta) \cap B (x, \varepsilon),
\end{equation}
 has a diffeomorphism type that does not depend on  $\delta$,
$\eta$ and $\varepsilon$. 
The characteristic mapping of the fibration
induces on $F_x$
an automorphism
which is defined up to homotopy,
the monodromy $M_x$.
In particular  the cohomology groups
$H^i (F_x, \Qq)$ are endowed with
an automorphism $M_x$, and for any integer $m$
one can consider the Lefschetz numbers
\begin{equation}
\Lambda (M_x^m) =  \tr (M_x^m ; H^{\bullet}(F_x, \Qq))
= \sum_{i \geq 0} (-1)^i\tr (M_x^m ; H^{i}(F_x, \Qq)).
\end{equation}

In  \cite{ACampo1}, A'Campo proved that if $x$ is a singular point of $f^{-1} (0)$, then
$\Lambda (M_x^1) = 0$ and this was later generalized by Deligne to the statement that
$\Lambda (M_x^m) = 0$ for $0 < m < \mu$, with $\mu$ the multiplicity of $f$ at $x$, cf. \cite{ACampo2}.

In \cite{DLtop}, Denef and  Loeser proved that $\Lambda (M_x^m)$ can be  expressed in terms of Euler characteristics of arc spaces as follows.
For any integer $m \geq 0$, let  $\mathcal{L}_m (X)$ denote the space of arcs modulo $t^{m+1}$ on $X$:
a $\mathbb{C}$-rational point of $\mathcal{L}_m (X)$ corresponds to a $\mathbb{C} [t] / t^{m + 1}$-rational point of X, cf. \cite{DLinv}.
Consider the locally closed subset  $\mathcal{X}_{m, x}$ of $\mathcal{L}_m (X)$ 
\begin{equation}\label{1.1.4}
\mathcal{X}_{m, x} =
\{\varphi \in \mathcal{L}_m (X); f (\varphi) = t^m \mod t^{m + 1}, \varphi (0) = x\}.
\end{equation}

Note that $\mathcal{X}_{m, x} $ can be viewed in a natural way as the set of closed points of a complex algebraic variety.

\begin{theo}[\cite{DLtop}]\label{mt}
For every $m \geq 1$,
\begin{equation}\label{eqmt}
\chi_c (\mathcal{X}_{m, x}) =
\Lambda (M_x^m).
\end{equation}
\end{theo}
Here $\chi_c$ denotes the usual Euler characteristic with compact supports.
Note that one recovers Deligne's statement as a corollary since $\mathcal{X}_{m, x} $ is empty for
$0 < m < \mu$.
The original proof in \cite{DLtop} proceeds as follows. One computes explicitly both sides of
(\ref{eqmt}) on an embedded resolution of the hypersurface defined by $f = 0$ and checks both quantities are equal.
The computation of the left-hand side  relies on the change of variable formula for motivic integration in
\cite{DLinv} and the one on
the right-hand side on A'Campo's formula in \cite{ACampo2}.
The problem of
finding a geometric proof of Theorem \ref{mt} not using resolution of
singularities is raised in \cite{seattle}.
The aim of this paper is to present such a proof.

\subsection{}Our approach uses \'etale cohomology of non-archimedean spaces and
motivic integration. Nicaise and Sebag introduced in
\cite{NS} the analytic Milnor fiber  $\mathcal{F}_x$ of the function $f$ at a point $x$ which is a rigid analytic space over $\mathbb{C} \llp t \rrp$.
Let
$\mathcal{F}_x^{an}$ denote its analytification in the sense of Berkovich.
Using a comparison theorem of Berkovich, they show that, for every $i \geq 0$, the  \'etale $\ell$-adic cohomology group
$H^i (\mathcal{F}_x^{an}  \widehat{\otimes} \, \widehat{\mathbb{C} \llp t \rrp^{\rm alg}}, \mathbb{Q}_{\ell} )$
is  isomorphic 
to $H^i (F_x, \mathbb{Q}) \otimes_{\Qq} \mathbb{Q}_{\ell}$.
Furthermore, these  \'etale $\ell$-adic cohomology groups are naturally endowed with an action of
the Galois group $\mathrm{Gal} (\mathbb{C} \llp t \rrp^{\rm alg} / \mathbb{C} \llp t \rrp)$
of the algebraic closure of $\mathbb{C} \llp t \rrp$, and
under this isomorphism
the action of the 
topological generator $(t^{1/n} \mapsto \exp (2 i \pi /n)t^{1/n})_{n \geq 1}$
of $\hat \mu (\mathbb{C}) = \mathrm{Gal} (\mathbb{C} \llp t \rrp^{\rm alg} / \mathbb{C} \llp t \rrp)$  corresponds to the monodromy $M_x$.

Another fundamental tool in our approach is provided by the theory of motivic integration developed in  \cite{HK} by Hrushovski and Kazhdan.
Their logical setting is that of 
 the theory
$\ACVF (0,0)$
of algebraically closed valued fields of equal characteristic zero,
with two sorts $\VF$ and $\RV$.
If $L$ is a field endowed with a valuation $v: L \to \G (L)$, with valuation ring $\Oo_L$ and maximal ideal $\Mm_L$, 
$\VF(L) = L$ and
$\RV(L)  = L^{\times} / (1 + \Mm_L)$.
Thus $\RV(L)$ can be inserted  in an exact sequence
\begin{equation}\label{rvext}
1 \to\kk^{\times}  (L) \to \RV \to \G (L)  \to 0
\end{equation}
with $\kk (L)$ the residue field of $L$.  Let us work with $\mathbb{C} \llp t \rrp$ as a base field.
One of the main result of  \cite{HK} is the construction of an isomorphism
\begin{equation}\label{keyseq}
\oint \colon K (\VF) \longrightarrow K (\RV [*]) / I_{\sp}
\end{equation}
between the Grothendieck ring $K (\VF)$ of definable sets in the $\VF$-sort 
and the quotient of  a graded version  $K (\RV [*])$ of  the Grothendieck ring of definable sets in the $\RV$-sort 
by an explicit ideal $I_{\sp}$.
At the Grothendieck rings level, the extension (\ref{rvext}) is reflected by the fact that
 $K (\RV [*])$ may be expressed as a tensor product 
 of the graded Grothendieck  rings $K (\G [*])$ and $K (\RES [*])$ for a certain sort $\RES$.
 A precise definition of $\RES$ will be given in \ref{ss2.2}, but let us say that variables  in the $\RES$ sort 
 range not only over the residue field but also over certain torsors over the residue field so that definable sets in 
the $\RES$ sort are twisted versions of constructible sets over the residue field.
This reflects the fact that the extension (\ref{rvext}) has no canonical splitting. 
Furthermore, there is a canonical isomorphism between  a quotient $\mathord{!K}(\RES)$  of the Grothendieck ring 
$K(\RES)$ and $K^{\hat \mu} (\Var_{\mathbb{C}})$,
the Grothendieck ring of complex algebraic varieties with $\hat \mu$-action, as considered in \cite{barc} and \cite{seattle}.
Let
 $[\mathbb{A}^1]$ denote the class of the affine line.
In \cite{HK}
a canonical morphism
\begin{equation}
\GEU \colon  K (\VF)  \longrightarrow \mathord{!K}(\RES) / ( [\mathbb{A}^1] - 1).
\end{equation}
is constructed. We shall make essential use of that construction, which is
 recalled in detail  in  \ref{2.5}. It roughly corresponds to
applying the o-minimal Euler characteristic to the $\Gamma$-part of the product decomposition 
of the right-hand of (\ref{keyseq}). 
Denote by  $K (\hat \mu\mathrm{-Mod})$ the 
 Grothendieck ring of the category of finite dimensional $\Qq_{\ell}$-vector spaces with  $ \hat{\mu} $-action. 
There is a canonical morphism
$K^{\hat \mu} (\Var_{\mathbb{C}}) \to K (\hat \mu\mathrm{-Mod})$ induced by taking the alternating sum of cohomology with compact supports
from which one derives a morphism
\begin{equation}\eu_{\text{\'et}} \colon \mathord{!K}(\RES) / ( [\mathbb{A}^1] - 1) \longrightarrow K (\hat \mu\mathrm{-Mod}).
\end{equation}

Our strategy is the following. Instead of trying to prove directly a Lefschetz fixed point formula for objects of
$\VF$, that are infinite dimensional in nature when considered as objects over $\mathbb{C}$, we take advantage of the morphism
$\GEU$ for reducing to finite dimensional spaces.
To this aim,
using \'etale cohomology of Berkovich spaces, developed by Berkovich in \cite{berket},
we construct a natural ring morphism
\begin{equation}
\etEU \colon
 K (\VF) \longrightarrow K (\hat \mu\mathrm{-Mod})
 \end{equation}
and we prove a  key result, Theorem
 \ref{theocompat}, 
which states that
the diagram
\begin{equation}\xymatrix{
K (\VF) \ar[rr]^{  \GEU} \ar[dr]_{ \etEU}& &\mathord{\mathord{!K}}(\RES) / ( [\mathbb{A}^1] - 1) \ar[dl]^{\eu_{\text{\'et}}}\\
&K (\hat \mu\mathrm{-Mod})& }
\end{equation}
is commutative.
Using this result, we are able to reduce the proof of Theorem \ref{mt}
to a classical statement, the Lefschetz fixed point
theorem
for finite order automorphisms acting on
complex algebraic varieties (Proposition \ref{lfp}).

Since our approach makes no use of resolution of singularities, 
it would be  tempting to try extending  it   to situations in positive residue characteristic. In order to do that, a necessary prerequisite would be to find the right extension of
the results of \cite{HK} beyond equicharacteristic $0$.

\subsection{}Using the same circle of ideas, we also obtain several new results and constructions  dealing with
the motivic Serre invariant and the motivic Milnor fiber.

More precisely,
in section \ref{secserre}, we  explain the connexion
between the morphism $\GEU$ and the motivic Serre invariant of
\cite{LS}.  We show in Proposition \ref{compserre}
that if $X$ is a smooth proper algebraic variety over
$F \llp t \rrp$ with $F$ a field of characteristic zero, with base change
$X (m)$ over $F \llp t^{m^{-1}}\rrp$,
then the motivic Serre invariant $S (X (m))$
can be expressed in terms of the part of $\GEU (X)$ fixed by
the $m$-th power of a topological generator of $\hat \mu$.
This allows in particular to provide a proof of a fixed point theorem originally proved by Nicaise and Sebag in  \cite{NS}
that circumvents the use
of resolution of singularities.

In section \ref{recov} we show how  one  can recover the motivic zeta function
and the motivic Milnor fiber of \cite{motivic} and  \cite{barc}, after inverting the elements $1 - [\mathbb{A}^1]^i$, $i \geq 1$,  from a single class
in the measured Grothendieck  ring of definable objects over $\VF$,
namely the class
of the set $\mathcal{X}_x$ of points $y$ in $X (\mathbb{C} \llb t \rrb)$
such that
$\rv f(y) = \rv(t)$ and $y(0) = x$.
This provides a new construction of the motivic Milnor fiber that  seems quite useful.
It  has already been used by L\^e Quy Thuong \cite{thuong} to prove an
integral identity conjectured by Kontsevich and Soibelman in their work on 
motivic
Donaldson-Thomas invariants \cite{KS}.

 \bigskip

We are grateful to Antoine Chambert-Loir, Georges Comte, Antoine Ducros, Johannes Nicaise, Michel Raibaut and Yimu Yin for very useful comments and exchanges.

\medskip{During the preparation of this paper, the research of the authors has been partially supported by the European Research Council under the European Union's Seventh Framework Programme (FP7/2007-2013) /  ERC Grant  agreement no. 291111 and ERC Grant agreement no. 246903/NMNAG.

\tableofcontents

\section{Preliminaries on Grothendieck rings of definable sets, after \cite{HK}} 
\subsection{}We shall consider the theory
$\ACVF (0,0)$
of algebraically closed valued fields of equal characteristic zero,
with two sorts $\VF$ and $\RV$. 
This will be  more suitable here than the more classical signature with three sorts $(\VF, \G, \kk)$. 
The language on $\VF$ is the ring language, and
the language on $\RV$ consists of abelian group operations $\cdot$ and $(\cdot)^{-1}$,
a unary predicate $\kk^{\times}$ for a subgroup,
an operation
$+ \colon \kk^2 \to \kk$, where $\kk$ is $\kk^{\times}$ augmented by a symbol zero,
and a function symbol $\rv$ for a function $ \VF^{\times} \to \RV$.
Here, $ \VF^{\times}$ stands for  $\VF \setminus \{0\}$. 

Let $L$ be a valued field, with valuation ring $\Oo_L$
and maximal ideal $\Mm_L$.
We set
$\VF(L) = L$,
$\RV(L)  = L^{\times} / (1 + \Mm_L)$,
$\Gamma (L) = L^{\times} / \Oo_L^{\times}$
and $\kk (L) = \Oo_L / \Mm_L$.
We have an exact sequence
\begin{equation}
1 \to\kk^{\times} \to \RV \to \G \to 0,
\end{equation}
where we view $\G$ as an imaginary sort. 
We denote by
$\rv \colon \VF^{\times} \to \RV$, $\val \colon \VF^{\times} \to \G$ and
$\val_{\rv} \colon \RV \to \G$ the natural maps.

\subsection{}\label{ss2.2}Fix a base structure $L_0$ which is a nontrivially valued field.
We shall view $L_0$-definable sets as functors from the category of  valued field extensions of $L_0$  with no morphisms except the identity to the category of sets.
For each $\gamma \in \mathbb{Q} \otimes \G (L_0)$,
we consider the definable set $V_{\gamma}$ 
\begin{equation}
L \longmapsto V_{\gamma} (L) = \{0\} \cup \{x \in L ;  \val (x) = \gamma\} /  ( 1 + \Mm_L)
\end{equation}
on valued field extensions $L$ of $L_0$.
Note that
when $\gamma - \gamma' \in \G (L_0)$, $V_{\gamma} (L)$ and $V_{\gamma'} (L)$ are definably isomorphic.
For  $\bar \gamma = (\gamma_1, \ldots, \gamma_n) \in
(\mathbb{Q} \otimes \G (L_0))^n$
we set
$V_{\bar \gamma} = \prod_i V_{\gamma_i}$.
By a $\bar \gamma$-weighted monomial, we mean
an expression $a_{\nu} X^{\nu} = a_{\nu} \prod_i X_i^{\nu_i}$
with $\nu = (\nu_1, \ldots, \nu_n) \in \mathbb{N}^n$ a multi-index,
such that
$a_\nu$ is an $L_0$-definable element of $\RV$ with
$\val_{\rv} (a_{\nu}) + \sum_i \nu_i \gamma_i = 0$.
A $\bar \gamma$-polynomial is a finite sum of
$\bar \gamma$-weighted monomials. Such a $\bar \gamma$-polynomial $H$
gives rise to a function $H \colon V_{\bar \gamma} \to \kk$ so we can consider
its zero set $Z (H)$. The intersection of finitely many  such sets is called a generalized algebraic variety over the residue field.
The generalized residue structure
$\RES$
consists of the residue field, together with the collection
of
the definable sets $V_{\gamma}$, for
$\gamma  \in
\mathbb{Q} \otimes \G (L_0)$,
and
the functions $H \colon V_{\bar \gamma} \to \kk$ associated to each
$\bar \gamma$-polynomial.

\subsection{}\label{azer}If $S$ is a sort, we write  $S^*$ to mean $S^m$, for some $m$.
We shall view varieties over $L_0$ as definable sets over $L_0$. We
denote by $\VF [n]$ the category of definable subsets of $n$-dimensional varieties over $L_0$.
By Lemma 8.1 of \cite{HK} this category is equivalent to the category whose objects are
the definable subsets $X$ of $\VF^* \times \RV^*$  such that there exists a definable
map $X \to \VF^n$ with finite fibers. By abuse of notation we shall sometimes also denote by $\VF [n]$ that category.

We denote by $\RV [n]$ the category of definable pairs
$(X, f)$ with $X \subset \RV^*$ and $f \colon X \to \RV^n$ a  definable map with finite fibers
and by
$\RES [n]$ the full subcategory consisting of  objects with $X$ such that
$\val_{\rv} (X)$ is finite (which is equivalent to the  condition that $X$ is isomorphic to a definable subset of $\RES^*$).
By Remark 3.67 of \cite{HK}, the forgetful map $(X, f) \mapsto X$ induces an equivalence of categories between $\RV[n]$ and the
category of all definable subsets of $\RV^*$  of $\RV$-dimension $\leq n$, that is, such that there
exists a definable map with finite fibers to $\RV^n$.
Nonetheless, the morphism $f$ will be useful for defining $ \mathbb{L}$ in \ref{2.4}.

Let $A = \G (L_0)$ or more generally any ordered abelian group.  One defines $\G [n]$
as the category whose objects are
 finite disjoint union of subsets of $\G^n$ defined by linear
equalities and inequalities with $\mathbb{Z}$-coefficients and parameters in $A$. Given objects $X$ and $Y$  in
 $\G [n]$, a morphism $f$ between $X$ and $Y$ is a bijection such that there exists a finite partition  of $X$
 into objects $X_i$ of
 $\G [n]$, such that the restriction of $f$ to $X_i$ is of the form
 $x \mapsto M_i x + a_i$
 with $M_i \in \mathrm{GL}_n (\mathbb{Z})$ and  $a_i \in A^n$.
 We define 
$\G^{\rm{fin}} [n]$ as the full subcategory of 
$\G [n]$ consisting of finite sets.

We shall consider the categories
\begin{equation}\RV[\leq
n] = \bigoplus_{0 \leq k \leq n} \RV [k],\end{equation}
\begin{equation}\RV [*] = \bigoplus_{ n \geq 0} \RV [n],\end{equation} 
\begin{equation}\RES [*] = \bigoplus_{ n \geq 0} \RES [n],\end{equation}
\begin{equation}
 \G [*] = \bigoplus_{n \geq 0} \G [n]\end{equation}
 and
\begin{equation}
 \G^{\rm{fin}} [*] = \bigoplus_{n \geq 0} \G^{\rm{fin}} [n].\end{equation}

Let $C$ be any of the symbols  $\RV$, $\RES$, $\G$  and $\G^{\rm{fin}}$.
We shall denote by $K_+ (C [n])$, $K_+ (C [*])$, resp. $K (C[n])$, $K (C [*])$,
the Grothendieck monoid, resp. the Grothendieck group, of the corresponding categories as defined in \cite{HK}.
The Grothendieck monoid
$K_+ (C[*])$ decomposes as a direct sum
 $K_+ (C[*])= \oplus_{0 \leq  n} K_+ (C[n])$ and admits a natural structure
 of graded semi-ring with $K_+ (C[n])$ as degree $n$ part.
 Similarly, $K(C[*])$  admits a natural structure
 of graded ring with $K (C[n])$ as degree $n$ part.
The Grothendieck monoid $K_+ (\RV [n])$ is isomorphic to the Grothendieck monoid of definable
subsets $X$ of $\RV^*$ of $\RV$-dimension $\leq n$.

One also considers $K_+ (\VF)$, resp.
$K (\VF)$,
the Grothendieck semi-ring, resp. the Grothendieck ring, of the category
of definable subsets of $L_0$-varieties of any dimension. The product is induced by cartesian product
and $K_+ (\VF)$ and $K (\VF)$ are filtered by dimension.
By Lemma 8.1 of \cite{HK}, $K_+ (\VF)$, resp.
$K (\VF)$,  can be identified with the Grothendieck semi-ring, resp. the Grothendieck ring, of the category
of definable subsets $X$ of $\VF^* \times \RV^*$  such that there exists, for some $n$, a definable
map $X \to \VF^n$ with finite fibers.
Similarly, one denotes by $K_+ (\RV)$,
$K (\RV)$, $K_+ (\RES)$, and $K(\RES)$, the Grothendieck semi-rings and  rings of the categories of definable subsets of $\RV^*$ and
$\RES^*$, respectively.

The mapping
$X \mapsto \val_{\rv}^{-1} (X)$
induces a functor $\G[n] \to \RV [n]$,
hence a morphism
$K_+ (\G [n]) \to
K_+ (\RV [n])$
which restricts to a morphism
$K_+ (\G^{\rm{fin}} [n]) \to
K_+ (\RES [n])$.
We also have a morphism $K_+ (\RES [n]) \to
K_+ (\RV [n])$
induced by the inclusion functor
$\RES[n] \to \RV [n]$.
There is a unique morphism of graded semi-rings
\begin{equation}\label{deco}
\Psi \colon K_+ (\RES [*]) \otimes_{K_+ (\G^{\rm{fin}}[*])}
K_+ (\G [*]) \longrightarrow
K_+ (\RV [*])
\end{equation}
sending
$[X] \times [Y]$ to
$[X \times \val_{\rv}^{-1} (Y)]$, for
$X$ in $\RES [m]$ and $Y$ in $\G[n]$ and it is proved in  Corollary 10.3 of \cite{HK} that 
$\Psi$ is an isomorphism.

\subsection{}\label{2.4}One defines
\begin{equation}\label{bbL}
\mathbb{L} \colon \mathrm{Ob} \RV [n] \longrightarrow \mathrm{Ob} \VF [n]
\end{equation}
by sending a definable pair
$(X, f)$ with $X \subset \RV^*$ and $f \colon X \to \RV^n$ a  definable map with finite fibers
to
\begin{equation}\mathbb{L}  (X, f) = \{(y_1, \ldots, y_n, x) \in (\VF^{\times})^n \times X ; (\rv (y_i)) = f (x)\}.\end{equation}
Note that by Proposition 6.1 of \cite{HK}, the isomorphism class of $\mathbb{L}  (X, f)$ does not depend on $f$, so we shall sometimes
write $\mathbb{L}  (X)$ instead of $\mathbb{L}  (X, f)$.
This mapping induces a morphism of filtered semi-rings
\begin{equation}
\mathbb{L} \colon K_+ (\RV [*]) \longrightarrow K_+ (\VF)
\end{equation}
sending the class of an object $X$ of
$ \RV [n]$ to the class of
$\mathbb{L} (X)$.

If $X$ is a definable subset of $\RV^n$, we denote by $[X]_n$ the class of 
$(X, \rm{Id})$ in $K_+ (\RV [n])$ or in $K(\RV [n])$. Similarly, if $X$ is a definable subset of $\RES^n$ or $\Gamma^n$, we denote by $[X]_n$ the class of 
$X$ in $K_+ (\RES [n])$ and $K_+ (\G [n])$, respectively, or in the corresponding Grothendieck ring. 
In particular, we can assign to the point $1 \in k^* \subset \RV$ a class
$[1]_1$    in $K_+ (\RV [1])$, and the point of $\RV^0$ a class   $[1]_0$ in  $K_+ (\RV [0])$.
Set
$\RV^{> 0} =\{x \in \RV; \val_{\rv} (x) > 0\}$.
Observe the identity $\mathbb{L}([1]_1) = \mathbb{L}([1]_0) + \mathbb{L}([\RV^{> 0}]_1)$ in $K_+(\VF)$;
the left hand side is the open ball $1 +\mathcal{M}$, while the right hand side is $(0)+ (\mathcal{M} \setminus (0))$.
Let $I_{\sp}$ be the semi-ring congruence
 generated by the relation $[1]_1 \sim [1]_0 + [\RV^{> 0}]_1$.
By Theorem 8.8 of \cite{HK}, $\mathbb{L}$ is surjective
with kernel $I_{\sp}$.
Thus, by inverting $\mathbb{L}$,
one gets a canonical isomorphism of filtered semi-rings
\begin{equation}\label{mainiso}\oint\colon K_+ (\VF) \longrightarrow K_+ (\RV [*]) / I_{\sp}.\end{equation}

\subsection{}\label{2.5}Let $I!$ be the ideal of
$K (\RES [*])$ generated by the differences
$[\val_{\rv} \inv(a)]_1-   [\val_{\rv} \inv(0)]_1$ where  $a $ runs over $ \G(L_0) \otimes \Qq$.
We  denote by
$\mathord{!K} (\RES [*])$ the quotient of $K (\RES [*])$ by $I!$
and 
by
$\mathord{!K} (\RES [n])$ its graded piece of degree $n$ (note
that passing to the quotient by
$I!$ preserves the graduation). 
One defines similarly $\mathord{!K} (\RES)$.


Let us still denote by $I_\sp$ the ideal in $K (\RV [*]) $ generated  by the similar object of $K_+ (\RV [*])$.
We shall now recall the construction of group morphisms
\begin{equation}\mathcal{E}_n  \colon K(\RV[\leq n])/I_\sp \longrightarrow  \mathord{!K}(\RES[n])\end{equation}
and
\begin{equation}\mathcal{E}'_n\colon K(\RV[\leq n])/I_\sp \longrightarrow  \mathord{!K}(\RES[n])\end{equation}
given in  Theorem 10.5 of \cite{HK}.

The morphism $\mathcal{E}_n$ is induced by the group morphism
\begin{equation}\gamma \colon \bigoplus_{m \leq n} K (\RV[m]) \longrightarrow \mathord{!K}(\RES[n])
\end{equation}
given by
\begin{equation}\gamma = \sum_m \beta_m \circ \chi [m],\end{equation}
with $\beta_m \colon  \mathord{!K}(\RES[m]) \to  \mathord{!K}(\RES[n])$
given by $[X] \mapsto [X \times \mathbb{A}^{n - m}]$
and $\chi [m]    \colon  K(\RV [m])  \to  \mathord{!K}(\RES[m])$ defined as follows.
The isomorphism (\ref{deco})
induces an isomorphism
\begin{equation}\label{deco2}
K(\RV [m])  \simeq \oplus_{1 \leq \ell \leq m}
K(\RES [m - \ell]) \otimes_{K (\G^{\rm{fin}})} K (\G [\ell]),
\end{equation}
and $\chi [m]$ is defined as $\oplus_{1 \leq \ell \leq m} \chi_{\ell}$
with
$\chi_{\ell}$ sending
$a \otimes b$ in $K(\RES [m - \ell]) \otimes_{K (\G^{\rm{fin}})} K (\G [\ell])$
to
$\chi (b) \cdot [\mathbb{G}_m]^{\ell} \cdot a$,
where
$\chi \colon K  (\G [\ell]) \to \mathbb{Z}$ is the o-minimal Euler characteristic (cf. Lemma 9.5 of \cite{HK}).
Here
 $\mathbb{G}_m$ denotes the multiplicative torus of the residue field, thus $[\mathbb{G}_m] = [\mathbb{A}^1] - 1 $.

The definition of $\mathcal{E}'_n$ is similar, replacing $\beta_m$ by the map $[X] \mapsto   [X] \times [1]^{n - m}_1$ and $\chi$ by
the ``bounded'' Euler characteristic 
$\chi'  \colon K  (\G [\ell]) \to \mathbb{Z}$ (cf. Lemma 9.6 of \cite{HK})
given by $\chi' (Y) = \lim_{r \to \infty} \chi (Y \cap [-r, r]^n)$
for $Y$ a definable subset of $\G^n$.

We will now consider $\mathord{!K}(\RES[n])$ modulo the ideal  of multiples of  the class of $[\mathbb{G}_m]_1$, which we denote by
$\mathord{!K}(\RES[n]) / [\mathbb{G}_m]_1$.
By the formulas (1) and (3)  in Theorem 10.5 of
\cite{HK}
the morphisms $\mathcal{E}_n$ and $\mathcal{E}'_n$  induce the same morphism
\begin{equation}E_n \colon K(\RV[\leq n])/I_\sp \longrightarrow  \mathord{!K}(\RES[n]) / [\mathbb{G}_m]_1.
\end{equation}
These morphisms are compatible, thus passing to the limit one gets
a morphism
\begin{equation}E \colon K(\RV[*])/I_\sp \longrightarrow  \mathord{!K}(\RES) / ( [\mathbb{A}^1] - 1).\end{equation}
In fact, the morphism $E$ is induced from both the morphisms $\mathcal{E}$ and $\mathcal{E}'$
from (2) and (4) in Theorem 10.5 of \cite{HK}.

 The morphism $E$ maps $[\RV^{>0}]_1$ to $0$,
and $[X]_k$ to $[X]_k$ for $X \in \RES[k]$.  Composing $E$ with the morphism
$K (\VF) \to
K(\RV[*])/I_\sp$ obtained by groupification of the morphism $\oint$ in (\ref{mainiso})
one gets a ring morphism
\begin{equation}\label{EU}
\GEU \colon  K (\VF)  \longrightarrow  \mathord{!K}(\RES) / ( [\mathbb{A}^1] - 1).
\end{equation}

\subsection{}The rest of this section is not really needed; it shows however that the introduction of  Euler characteristics for $\G$ can be bypassed
in the construction of $\GEU$.

Let $\val=\val_{\rv}$ denote  the canonical map $\RV \to \G$.
Let $I_\G'$ be the ideal of $K(\RV[*])$ generated by all classes $[\val \inv(U) ]_m$, for $U$ a definable subset of $ \G^m$, $m \geq 1$, and
let $I_*$ be the ideal generated by $I_\G'$ along with $I_\sp$.  Since $[\RV^{>0}]_1 \in I_\G'$, the canonical generator $[\RV^{>0}]_1+[1]_0-[1]_1$ reduces, modulo
$I_\G'$, to $[1]_0 - [1]_1$, i.e. the different dimensions are identified.
 Thus $K(\RV[*])/I_* = K(\RV)/I_\G$, where on the right
we have the ideal of $K(\RV)$ generated by  all classes $[\val \inv(U) ]$, for any definable $U \subset \G^m$, $m \geq 1$, or equivalently just by $\val \inv(\{0\})$.

\begin{lem}
 The inclusion functor $\RES \to \RV$  induces an isomorphism
 \[\mathord{!K}(\RES) / ([\mathbb{A}^1] - 1) \longrightarrow K(\RV)/I_\G.\]
 \end{lem}

\begin{proof} This is already true even at the semi-ring level, as follows from Proposition 10.2 of \cite{HK}.   The elements $[\val \inv(U) ]$
of $K_+(\RV)$ are those of the form $1 \tensor b$ in the tensor product description, with $b \in K_+(\G[n])$, $n \geq 1$.
Moding out the tensor product $K_+(\RES) \tensor K_+(\G[*])$ by these elements we obtain simply
$K_+(\RES) \tensor K_+(\G[0]) \simeq K_+(\RES) $.  Now taking into account the relations of the
tensor product amalgamated over $K_+(\G^{\rm{fin}})$, namely $1 \tensor [\g]_1= [\rv \inv (\g)] \tensor [1]_0$,
as the left hand side vanishes, we obtain the relation $[\rv \inv(\g)] = 0$.  These are precisely the relations defining
$\mathord{!K}(\RES)$ (namely $[\rv \inv(\g)] = [\rv \inv(\g')]$) along with the relation  $\rv \inv(0)=0$ (i.e. $[\mathbb{A}^1] - 1= 0$).
\end{proof}

\begin{rema}
It is also easy to compute that the map  \begin{equation}{\mathcal E}\colon K(\RV[*]/I_\sp) \longrightarrow \mathord{!K}(\RES)[[\Aa^1]^{-1}]\end{equation}
from \cite{HK}, Theorem 10.5,
composed with  the natural map $ \mathord{!K}(\RES)[[\Aa^1]^{-1}] \to \mathord{!K}(\RES)/([\mathbb{A}^1] - 1) $,  induces the retraction $K(\RV)/I_\G \to \mathord{!K}(\RES)/([\mathbb{A}^1] - 1)$
above.
\end{rema}

\section{Invariant admissible transformations}
We continue to work in $\ACVF (0,0)$ over a base structure $L_0$ which is a nontrivially valued field.


\subsection{}
For     $\a \in \G (L_0)$, one sets  $\Oo \a = \{x: \val(x) \geq \alpha \}$,
and
 $\Mm \a = \{x: \val(x)> \alpha \}$.    For $x=(x^\prime,x^{\prime\prime}),y=(y^{\prime},y^{\prime \prime}) \in \VF^n \times \RV^m$, write $v(x-y) > \alpha$ if $x^{\prime}-y^{\prime} \in (\Mm \a)^n$.
 If $f$ is a definable function on a definable subset
$X$ of $\VF^n \times \RV^m$, say $f$ is $\a$-invariant,
resp. $ \alpha^+$-invariant,
if $f(x+y)=f(x)$ whenever $x,x+y \in X$ and $y \in (\Oo \alpha )^n$,
resp. $y \in (\Mm \alpha )^n$.
Say a definable set $Y$ is $\a$-invariant, resp. $ \alpha^+$-invariant,
if the characteristic function $1_Y\colon \VF^n \times \RV^m \to \{0,1\}$ is $\a$-invariant, resp. $ \alpha^+$-invariant.

Call a definable set of imaginaries {\em non-field} if it admits  no definable map onto a non-empty open disk (over parameters).  Any
imaginary set of the form $\mathrm{GL}_n / H$, where $H$ is a definable   subgroup of $\mathrm{GL}_n$ containing a valuative neighborhood of $1$,
has this property.   By \cite{HHK}, $\ACVF$ admits elimination of imaginaries to the level of certain ``geometric sorts"; these include the valued
field $K$ itself and certain other sorts of the form $\mathrm{GL}_n/H$ as above.  We may thus restrict  our attention to such  sorts in the lemma below.  Note that for a separable topological field $L$, $\mathrm{GL}_n(L)$ is separable
while $H(L)$ is an open subgroup, so $(\mathrm{GL}_n/H)(L)$ is countable.

\begin{lem}\label{1}Let $A$ be a set of imaginaries.
Let $X \subset \VF^n$  be an $A$-definable subset bounded and closed in the valuation topology.
 Let   $f\colon X \to  W$ be $A$-definable,
where $W$ is a non-field set of  imaginaries.
Fix $\alpha$ in $\Gamma (L_0)$.
 Then   there exists a  $\b \geq \a$, a
 $\b^+$-invariant $A$-definable map $g\colon X \to W$ such that for any $x \in X$, for some $y \in X$, $v(x-y)>\a$   and
$g(x)=f(y)$.
  \end{lem}

 \begin{proof}  We use induction on $\dim(X)$.  If $\dim(X)=0$ we can take $f=g$, and $\b$  the maximum of $\a$ and  the maximal valuative
 distance between two distinct points of $X$.  So assume $\dim(X)>0$.

Let us start by proving that there exists a relatively Zariski closed definable subset $Y \subset X$ such that  $\dim(Y)< \dim(X)$ and such that $f$ is locally constant on $X \setminus Y$.  
To do this, we work within the Zariski closure $\bar{X}$ of $X$.  We use both the Zariski topology and
the valuation topology on $\bar{X}$; when referring to the latter we use the prefix v.   It follows from quantifier-elimination that any definable
subset differs from a  v-open set by a set contained in a subvariety of $\bar{X}$ of dimension $< \dim(X)$.  In particular, a
 definable subset of $\bar{X}$ of dimension $\dim(X)$ must contains a v-open set.  Now
  the locus $Z$ where $f$ is locally constant is definable. Assume by contradiction that $Z$  does not contain  a Zariski dense open subset of $\bar{X}$, then its complement contains
 a non-empty v-open set $e$. Note that on every non-empty v-open definable subset of $e$, $f$ is non-constant, since otherwise it would intersect $Z$. It follows
 that the following property holds:

$(\ast)$ the Zariski closure of
 $e
\cap f ^{-1}(w)$ is of dimension $<n$ for every $w$ in $W$.

Thus, for any model of $\ACVF (0, 0)$, there exists
 $X' \subset \VF^n$  definable  bounded and closed in the valuation topology,   $f'\colon X' \to  W'$ definable
with $W'$ a non-field set of  imaginaries and a non-empty v-open definable subset $e'$ such that $(*)$ holds.
It follows that for $p$ large enough there exist such $X'$, $W'$, $f'$ and $e'$ defined over the algebraic closure of  $\mathbb{Q}_p$
such that $(*)$ holds. Take a finite extension $L$ of $\mathbb{Q}_p$ over which $X'$, $W'$, $f'$ and $e'$ are defined.
Since $W' (L)$ is countable and, by $(*)$, $f'^{-1}(w') \cap e' (L)$  is of measure zero, for each $w' \in W' (L)$, it follows
that  $e' (L)$ is of measure zero, a contradiction.

 \medskip

 By the inductive
 hypothesis, there exist ${\b'} \geq \a$,  a  ${\b'}^+$-invariant function $g_Y\colon Y \to W$ such that  for any
  $y \in Y$,
   for some $z \in Y$, $v(y-z) > \alpha$
   and
$g_Y(y)=f(z)$.

 Let $Y'=\{x \in X: (\exists y \in Y) (v(x-y)> \b' ) \}$.  One extends $g_Y$ to a function $g'$ on $Y'$ by defining $g'(x)=g_Y (y)$ where $y$ is an element of $Y$ such that $v(x-y)>\b'$.
 By the ${\b'}^+$-invariance of $g_Y$, this is well-defined.  Moreover, for any $y \in Y'$, there exists $z \in Y$ such that $v(y-z)> \a$
   and
$g'(y)=f(z) $.

For each $x$ in $X \setminus Y$, we denote by $\d(x)$ the valuative radius of the maximal open ball around
$x$
 contained in $X \setminus Y$ on which $f$ is constant. Since $X \setminus Y'$ is closed and bounded,  
  $\d$ is bounded  on $X \setminus Y'$ by Lemma 11.6 of \cite{HK}.
 Thus, there exists  $\b \geq \b'$ such that 
 if $x,x' \in X \setminus Y'$ and $x-x' \in \Mm \b$, then $f(x)=f(x')$.    We now define $g$ on $X$ by $g(x)=g'(x)$ for $x \in Y'$, and $g(x)=f(x)$ for $x \in X \setminus Y'$.  Note that if $x,x' \in X$ and
$v(x-x' ) > \b (\geq \b')$, then   either $x,x' \in Y'$ or $x,x' \in X \setminus Y'$; in both cases, $g(x)=g(x')$.  We have already seen
the last condition holds on $Y'$; it clearly holds for $x \in X \setminus Y'$, with $y=x$.
 \end{proof}

 We repeat here Corollary 2.29 of \cite{metastable}.

 \<{lem}\label{bdd-im}  Let $D$ be a $C$-definable set in $\ACVF$ that may contain imaginary elements.  Then the following are equivalent:
\begin{enumerate}
\item[(1)]There exists a definable surjective map $g\colon (\Oo /  \Oo \beta)^n \to D$.
\item[(2)]There is no definable  function $f\colon D \to \G$ with unbounded image.
\item[(3)]For some  $\beta_0 \leq 0 \leq \beta_1 \in \G(C)$, for any $e \in D$,
$e \in \dcl(C,   \Oo\beta_0 /  \Mm \beta_1) $.
\end{enumerate}
\>{lem}

A definable set $D$ (of imaginary elements) satisfying (1-3) will be called    {\em boundedly imaginary}.
 An infinite subset of the valued field can never be boundedly imaginary; a subset of the value group, or of $\G^n$, is boundedly imaginary iff it is bounded;
 a subset of $\RV^n$ is boundedly imaginary iff its image in $\G^n$ under the valuation map  is bounded (i.e. contained in a box $[-\g,\g]^n$).
 We shall say a subset of $\RV^n$ is bounded below if its image in $\G^n$ under the valuation map  is bounded below (i.e. contained in a box $[\gamma,\infty)^n$).

 \<{lem}    \label{1u} Let $T$ be a boundedly imaginary definable set.  Let $X \subset \VF^n \times T$, and, for $t \in T$, set $X_t = \{x: (x,t) \in X\}$.  Assume
 each $X_t$   is bounded and closed in the valuation topology. Let
 $W$ be a non-field set  of imaginaries
 and let   $f\colon X \to  W$ be a definable map. Fix
 $\alpha$ in $\Gamma (L_0)$.
 Then   there exist  $\b \geq \a$, a
 $\b^+$-invariant definable function $g\colon X \to W$ such that for any $t \in T$ and $x \in X_t$,  there exists $y \in X_t$, $v(x-y)>\a$   and
$g(x,t)=f(y,t)$.
  \end{lem}

\begin{proof}  For each $t$ we obtain, from Lemma \ref{1}, an $A(t)$-definable element $\b(t) \geq \a$, and a $\b(t)^+$-invariant $g_t\colon X_t \to W$, with the stated property.
 As $T$ is boundedly imaginary, $\b (t)$ is bounded on $T$ and  $\b = \sup_t \b(t) \in \G$.  For each $t$, the statement remains true with $\b(t)$ replaced by $\b$.  By the usual compactness / glueing argument, we may take
 $g_t$ to be uniformly definable, i.e. $g_t(x)=g(x,t)$. \end{proof}

 \subsection{}We now define an invariant analogue of  the admissible transformations of \cite{HK}, Definition 4.1.

Let $n \geq 1$ an integer and let $\beta = (\b_1,\ldots,\b_n) \in \G^n$.  Let $\VF^n/ \mathcal{O} \b = \prod_{1 \leq i \leq n} (\VF/  \mathcal{O} \b_i)$, and let $\pi = \pi_\b\colon \VF^n \to \VF^n/ \mathcal{O} \b$ be the natural map.
Also write $\pi(x,y) = (\pi(x),y)$ if $x \in \VF^n$ and $y \in \RV^m$.
Say $X \subseteq \VF^n \times \RV^m$ is $\b$-invariant if it is a pullback via $\pi_\b$; and that $f\colon \VF^n \times \RV^* \to \VF$ is $(\b,\a)$-covariant if it induces
a map $\VF^n/ \mathcal{O} \b \times \RV^* \to \VF/ \mathcal{O} \a$, via $(\pi_\b,\pi_\a)$.

\begin{Defi}\label{4def}Let $A$ be a base structure.
 Let $n \geq 1$ an integer and let $\beta = (\b_1,\ldots,\b_n) \in \G^n$. 
 \begin{enumerate}
\item[(1)] An elementary 
 $\b$-invariant  admissible transformation over $A$ 
 is a function of one of the following types:
  \begin{enumerate}
\item[(i)]a function  $\VF^n \times \RV^m \to \VF^n \times \RV^m$ of the form
\begin{equation*}
(x_1, \dots, x_n, y_1, \ldots, y_m) \longmapsto (x_1, \ldots, x_{i - 1}, x_i + a, x_{i + 1}, \ldots, x_n, y_1, \dots, y_m)
\end{equation*}
with $a = a(x_1,\ldots,x_{i-1},y_1,\ldots,y_l)$ a $(\beta, \beta_i)$-covariant $A$-definable function and $m\geq 0$ an integer.
\item[(ii)]a function $\VF^n \times \RV^m \to \VF^n \times \RV^{m + 1}$ of the form
\begin{equation*}(x_1,\ldots,x_n,y_1,\ldots,y_l) \longmapsto
(x_1,\ldots,x_n,y_1,\ldots,y_l,h(x_i))
\end{equation*}with $h$ an $A$-definable $\b_i$-invariant function $\VF \to \RV$ and $m\geq 0$ an integer.
\end{enumerate}
\item[(2)]  Let $m$ and $m'$ be non negative integers.
 A function  $\VF^n \times \RV^m \to \VF^n \times \RV^{m'}$ 
 is called $\b$-invariant  admissible transformation over $A$ if it is the composition of elementary 
 $\b$-invariant  admissible transformations over $A$.
 \item[(3)]  Let $\mathcal{C}'_{A}(\b)$ be the category whose objects
are triples $(m, W, X)$ with $m \geq 0$ an integer, $W$ a boundedly imaginary definable set   contained
 in $\RV^m$ and $X$ a definable subset of $\VF^n \times W$  such that $X_w$ is a bounded, $\b$-invariant subset of $\VF^n$, for every $w \in W$.
 A morphism
 $(m, W, X) \to (m', W', X')$ 
 in $\mathcal{C}'_{A}(\b)$ is a definable map
 $X \to X'$ which is the restriction of some $\b$-invariant  admissible transformation
 $\VF^n \times \RV^m \to \VF^n \times \RV^{m'}$.   We consider the full subcategory $\mC_A(\b)$ whose objects $X$
 satisfy the additional condition that the projection $X \to \VF^n$ has finite fibers.
  \item[(4)] Let   $(m, W, Z)$ be in $\mathcal{C}'_{A}(\b)$. We say $Z$ is  elementary if there exists an integer
  $m' \geq 0$, 
  a $\b$-invariant  admissible transformation
 $T: \VF^n \times \RV^m \to \VF^n \times \RV^{m'}$, a definable subset 
 $H$ of $\RV^{m'}$,
 and a map $h \colon \{1, \ldots, n \} \to\{1, \ldots, m' \}$
 such that
  \begin{equation*}
 T (Z) = \{(a, b) \in \VF^n \times H; \rv (a_i) = b_{h (i)}, \text{for} \; 1 \leq i \leq n \}.
 \end{equation*}
\end{enumerate}
 \end{Defi}
 
 If $\beta = (\b_1,\ldots,\b_n)$ and 
 $\beta' = (\b'_1,\ldots,\b'_n)$ are in $\G^n$, we write 
 $\beta \geq \b'$ if 
 $\beta_i \geq \b'_i$ for every $1 \leq i \leq n$.
 If $\beta \geq \beta'$, we have a natural embedding of $\mathcal{C}'_{A}(\b)$ as a (non-full) subcategory of $\mathcal{C}'_{A}(\b')$.
We denote by $\mC'_{A}$, resp. $\mC_A$, the direct limit over all $\b$ of the categories $\mC'_{A}(\b)$, resp. $\mC_{A}(\b)$.

The following such Proposition is an analogue of  Proposition 4.5 of \cite{HK}  in the category $\mC_{A}$.

 \begin{prop} \label{2} Let $F$ be a subset of a model of $\ACVF (0, 0)$ such that, for each  $\g \in \G(F)$, there exists $f \in \VF(F)$
such that $\val(f)>\g$. 
We work in  $\ACVF_F$. Let $\a \in \G^n$ and let 
$(\ell, W, X)$ be an object in $\mC_F (\a)$.
 There exists
$\b \geq \a$ such that 
$X$ is a  Boolean combination  of finitely many $\b$-invariant definable subsets $Z$ which are elementary in the sense
of  \textup{Definition \ref{4def} (4)}. Furthermore, if the projection $X \to \VF^n$ has finite fibers, one may assume that for each such $Z$, the projection $H \to \RV^{n}$ given by
 $b \mapsto (b_{h(1)}, \ldots, b_{h (n)})$ has finite fibers.
\end{prop}

\begin{proof}Note that the  hypothesis on the base set $F$ is preserved if we move from $F$ to $F(w)$, where $w$ lies in a boundedly imaginary definable set.
This permits the inductive argument  below to work.

We now explain how to adapt the  proof of Proposition 4.5 of \cite{HK} to the present setting.
We first adapt Lemma~4.2 of  \cite{HK}. In that lemma,
if $X$ is $\a^+$-invariant, the proof gives $\a^+$-invariant sets $Z_i$ and transformations $T_i$.
As stated there, the $\RV$ sets $H_i \subset \RV_\infty^{\ell_i}$ are bounded below, since the assumption made on $X$ implies that
$X \times W \subset B \times W$, for some bounded $B \subset \VF^n$. However we need to modify the proof there in order to obtain boundedly imaginary sets.  This occurs 
where $X$ is a ball around $0$, namely in
cases 1 and 2 in the proof of Lemma~4.2 of \cite{HK}.  In these cases  choose a definable $f \in \VF$ such that $\val(f)$ is bigger than the radius of $X$.  Let $Y$ be an open ball around
$0$ of radius $\val(f)$. Then $X \setminus Y$ is the pullback from $\RV$ of a boundedly imaginary set.  As for $Y$ we may move it to $f+Y$,
which is the pullback from $\RV$ of a single element.  It is at this point that we require Boolean combinations instead of  unions.

Next, let us adapt the argument in the proof of Proposition 4.5 of \cite{HK}.
Given a definable map $\pi\colon X \to U$, with
$U$ a
definable subset of $\VF^{n- 1} \times V$ with $V$ a boundedly imaginary definable set   contained
 in $\RV^{\ell}$, such that $U_v$ is a bounded subset of $\VF^{n - 1}$, for every $v \in V$,
such that
$X$, $U$ and $\pi$ are all $\a^+$-invariant, we obtain a partition and transformations of $X$ over $U$,
such that each fiber becomes an $\RV$-pullback, and each piece of each fiber is   $\a^+$-invariant.
Note that the fiber above $u$ depends only on $u + (\Mm \a)^{n - 1}$.  Note also that $U$, being $\a^+$-invariant, is clopen in the
valuation topology.   Using Lemma \ref{1u},  we may modify the partition and the admissible transformations so as to be $\b^+$-invariant,
for some $\b \geq \a$.
With this, the inductive proof of \cite{HK}, Proposition 4.5 goes through to give the invariant result.
\end{proof}


\section{Working over $F \llp t \rrp$}\label{sec4}

\subsection{}\label{4.1}We now work over the base  field $L_0= F\llp t \rrp$, with $F$ a trivially valued algebraically closed field of characteristic zero and $\val(t)$ positive and denoted by $1$.
  Then the sorts of $\RES$ are the $\kk$-vector spaces $V_{\gamma} = \{x \in \RV:  \val_{\rv}(x) = \gamma\} \union \{0\}$, for $\gamma \in \Qq$.   
  Let $k \in \Zz$ and $m$ a positive integer.
  Since we have a definable bijection
$V_{k/m} \to V_{(k+m)/m}$ given by multiplication by $\rv(t)$, it suffices to consider $V_{k/m}$ with $0 \leq k < m$ and $m$ a positive integer.

The Galois group of $F\llp t \rrp^{\rm alg} / F\llp t \rrp$ may be identified with
the group
$\hat \mu = \liminv \mu_n$ of roots of unity and it acts on $\RES$ by automorphisms. On $V_{k/n}$, a primitive $n$-th root  of $1$, say $\zeta$, acts by multiplication
by $\zeta^k$.     We have an induced action on $K(\RES)$.  The classes $[V_{k/n}]$
are fixed by this action; and so an action is induced on $\mathord{!K}(\RES)$.

Given a positive integer $m$, let $\RES_{m^{-1} \Zz}$ denote the
sorts of $\RES$ fixed by $\hat{\mu}^m$, the kernel of $\hat \mu \to \mu_m$, namely, $V_{k/m}$ for $k \in \Zz$.

Projection on  $\RES_{m^{-1} \Zz}$ provides a canonical morphism
\begin{equation}\Delta_m \colon K_+ (\RES) \longrightarrow  K_+ (\RES_{m^{-1} \Zz})\end{equation}
inducing
\begin{equation}\label{delta}\Delta_m \colon \mathord{!K} (\RES) / ([\mathbb{A}^1] - 1)\longrightarrow \mathord{!K} (\RES_{m^{-1} \Zz}) / ([\mathbb{A}^1] - 1),\end{equation}
where $\mathord{!K} (\RES_{m^{-1} \Zz})$ is defined similarly as was  $\mathord{!K} (\RES)$ in \ref{2.5}.
One denotes by
$\EU_{\Gamma, m}$
the morphism
\begin{equation}\label{Delta}\EU_{\Gamma, m} \colon
K (\VF) \longrightarrow
\mathord{!K}  (\RES_{m^{-1} \Zz}) / ([\mathbb{A}^1] - 1)\end{equation}
obtained by composing $\GEU$ in (\ref{EU}) and
$\Delta_m$ in  (\ref{delta}).

\medskip

The following statement is straightforward:

\begin{lem}\label{fixed}Let $r$ and $n$ be integers, let
$X$ be a definable subset of $\VF^r$, let
$Y$ a definable subset of $\RES^n$. 
Assume that
$\GEU([X]) = [Y]$.
Then, for any positive integer $m$,
 $\EU_{\Gamma, m}([X])$ is the class of the subset of $ Y$
fixed by $\hat{\mu}^m$. \qed
\end{lem}

\subsection{}\label{4.2}Inside a given algebraic closure of $F\llp t \rrp$, the field $K_m = F\llp t^{1/m} \rrp$ does not depend on a particular choice of $t^{1/m}$, and $\mu_m$ acts on it.
Let $\b \in \frac{1}{m} \Zz^n \subset \G^n$, and let $X \subset \VF^n  \times \RV^{\ell}$ be a $\b$-invariant  $K$-definable set such that the projection
$X \to \VF^n $ has finite fibers.
We assume
$X$ is contained in $\VF^n \times W$ with $W$ a boundedly imaginary definable subset   of
 in $\RV^{\ell}$, and that, for every $w \in W$,  $X_w$ is a bounded.
Thus,  $X_w$ is $\b$-invariant for each $w$ in $\RV^{\ell}$, the projection of $X$ to  $\G^{\ell}$  is contained in a cube $[-\alpha, \alpha]^{\ell}$, and the projection of  $X$ to $\VF^n$ is contained in $c \Oo^n$ for some $c$.
For notational simplicity, and since this is what we will use, we shall assume $X \subset \Oo^n \times \RV^{\ell}$.

  Then the $K_m$-points $X(K_m)$ are the pullback of some subset  $X[m;\b] \subseteq  \Pi_{i=1}^n F[t^{1/m}] / t^{\b_i} \times \RV^{\ell}$; and the
   projection $X[m;\b] \to \Oo^n$ has finite fibers.

We can identify $F[t^{1/m}] / t^{N}$ with $\oplus _{0 \leq k < mN} V_{k/m} \cong \oplus _{0 \leq k < m} V_{k/m}^N$.
Also, if $Y$ is definable in  $\RV$ and $\val_{\rv}(Y) \subset [-\alpha,\alpha]$,
then
\begin{equation}Y(F\llp t^{1/m}\rrp) \subset \union \{V_\gamma:  \gamma \in m^{-1} \Zz \meet [-\alpha,\alpha]  \}.
\end{equation}
Thus $X[m; \beta]$ can be viewed as
a subset of the structure $\RES_{m^{-1} \Zz}$ (over $F$).
Here are  three ways to see it is definable. The first one is to say
it is definable in $(F\llp t^{1/m}\rrp,t)$; the induced structure on the sorts $V_{k/m}$ is the same as the structure induced
from $\ACVF$.
The second one is to remark that after finitely many invariant admissible transformations, $X$ becomes a set in standard form, a pullback from $\RV$.
These operations induce quantifier-free definable maps on the sets $X[m;\beta]$; so it suffices to take $X$ in standard form, and then
the statement is clear. Thirdly, in the structure $F \llp t \rrp^{\rm alg}$ with a distinguished predicate for $F$, it is clear that $F \llp t^{1/m} \rrp$ is definable
and so $X(F \llp t^{1/m} \rrp)$ is definable; and here too (cf. \cite{HK3}, Lemma 6.3) the induced structure on $F$ is just the field structure, and 
the induced structure on the sorts $V_{k/m}$ is the same as the structure induced
from $\ACVF$.



\begin{lem} Let $X$ be as above and let 
 $\b' $
 in $\frac{1}{m} \mathbb{Z}^n$, with $\b_i \leq \b'_i$, for every $1 \leq i \leq n$. 
 Then $[X[m;\b']]=[X[m;\b]]\times [\Aa^{n m \sum_i (\b'_i - \b_i)}]$
 in $\mathord{K}(\RES_{m^{-1} \Zz})$.
\end{lem}
\begin{proof}We shall  may assume $\b$ differs from $\b'$ in one coordinate, say the first, and that $\b'_1 = \b_1+ \frac{1}{m} $.
Consider the projection $X[m,\b'] \to X[m,\b]$.  Working over a parameter $t^{1/m}$, this is a morphism of $\mathrm{ACF}$-constructible sets, whose fibers
are $\Aa^1(\kk)$-torsors; so by Hilbert 90, there exists a constructible section.  Now this section may not be $\mu_m$-invariant, 
but after averaging the $\mu_m$-conjugates one finds a $\mu_m$-invariant section, which is $F \llp t \rrp$-definable.
It follows that $X[m;\b']=X[m;\b] \times \Aa^1$, as required.  
 \end{proof}

Thus the class of 
$[X[m;\b]]/ [\Aa^{\tcb{m (\sum _i \b_i) - n}}]$  in the localization $\mathord{K}(\RES_{m^{-1} \Zz})[[\Aa^1]^{-1}]$ does not depend on $\b$; let us 
denote it by
$\widetilde{X} [m]$.

We also denote by $X[m]$
the image of $[X[m;\b]]$ in $\mathord{!K}(\RES_{m^{-1} \Zz})/ ([\Aa^1] - 1)$, or in
$\mathord{!K}(\RES)/ ([\Aa^1] - 1)$,
which does not depend on $\b$.


Let $X$ be as before and let $f\colon X \to Y$ be a $\b$-invariant admissible bijection  in  $\mC(\b)$.
Since
    $f$ induces a bijection between $X[m;\b]$ and $Y[m;\b]$, it follows that $\widetilde{X}[m]=\widetilde{Y}[m]$ and $X[m]=Y[m]$.

\begin{prop}\label{3}
  Let $X$ be a $\b$-invariant $F \llp t \rrp$-definable subset of $\Oo^n \times \RV^{\ell}$, for some $\beta$.
Assume the projection $X \to \VF^n$ has finite fibers.
Then, for every $m \geq 1$,   $\EU_{\Gamma, m} (X) = X[m]$ as classes in $\mathord{!K}(\RES_{m^{-1} \Zz}) /
([\mathbb{A}^1] - 1)$.
\end{prop}

\begin{proof}Since both sides are invariant under the transformations of Proposition \ref{2},
we may assume by Proposition \ref{2} that
there exists a
definable boundedly imaginary
subset $H$ of $\RV^{\ell'}$
 and a map $h \colon \{1, \ldots, n \} \to\{1, \ldots, \ell' \}$
 such that
 \begin{equation}
 X = \{(a, b); b \in H, \rv (a_i) = b_{h (i)}, 1 \leq i \leq n \}
 \end{equation}
and the map  $r \colon H \to \RV^{n}$ given by $b \mapsto (b_{h(1)}, \ldots, b_{h (n)})$ is finite to one.
According to (\ref{deco}) we may assume that the class $[(H, r)]$ is equal to $\Psi ([W] \otimes [\Delta])$
with  $W$ in $\RES [\ell]$ and
$\Delta$ bounded in $\Gamma [n - \ell]$.
By induction on dimension and considering products, it is  enough to prove the result when $X$ is the lifting of an object of $\Gamma$ or  $\RES$.
Let us prove that the image of the canonical lift from $\G$
vanishes for both invariants.  In the case of $\EU_{\Gamma, m}$, the lift of any $Z \subset \G^q$, $q \geq 1$,
to $K(\RV)$  vanishes modulo $[\mathbb{A}^1]-1$.  In the case of $X[m]$, finitely many points of the value group of $K_m$
in the cube $[0,N]^n$ lie in $Z$; again for each such point, the class of $\mathord{!K}(\RES)$ lying above it is divisible by $[\mathbb{A}^1]-1$.  On the other hand
on $\RES$, both $\EU_{\Gamma, m}$ and $X[m]$ correspond to intersection with $\RES_{m^{-1} \Zz}$.   \end{proof}


\begin{cor}\label{4} Let $X$ be a smooth  variety over $F$, $f$ a regular function on $X$ and $x$ a closed point of
$f^{-1} (0)$. Let $\pi$ denote the reduction map $X (\Oo) \to X (F)$.
Let \[X_{t, x}= \{y \in X(\Oo); f(y)=t  \,\, \mathrm{and} \,\, \pi (y) = x \}\] and let  \[\mathcal{X}_x  = \{y \in X (\Oo); \rv f(y) = \rv(t) \,\, \mathrm{and} \,\, \pi (y) = x \}.\]  Then $\mathcal{X}_x$ is
 $\beta$-invariant for $\beta >0$, and, for every $m \geq 1$,
  $\EU_{\Gamma, m}(X_{t, x}) =\mathcal{X}_x[m]$ as classes in $\mathord{!K}(\RES) / ([\mathbb{A}^1] - 1)$. \end{cor}

\begin{proof}  The $\b$-invariance of $\mathcal{X}_x$ is
clear.
 Consider the canonical morphism
\[\oint\colon K_+ (\VF) \longrightarrow K_+ (\RV [*]) / I_{\sp}\] of (\ref{mainiso}).
For any $t'$ with $\rv(t')=\rv(t)$,  there is an automorphism over $F$ fixing $\RV$ that
sends $t$ to $t'$, thus
\[\oint [X_{t', x}] = \oint [X_{t, x}].\]
It follows that
\begin{equation}\oint([\mathcal{X}_x]) = \oint [ \union\{X_{t', x}: \rv(t')=\rv(t) \}] = (\oint [X_{t, x}])  \cdot e,
\end{equation}
where $e$ is the class of an open ball,
i.e. $e=[1]_1$.  Applying $\GEU$ we find that
$\GEU(\mathcal{X}_x) = \GEU(X_{t, x})$, and the statement follows from Proposition \ref{3}.
\end{proof}

\subsection{}\label{3.4}
Let $X$ be a quasi-projective variety over $F$. We say a $\hat \mu$-action is \emph{good} if it factorizes through some $\mu_n$-action, for some $n \geq 1$.
We denote by $K_+^{\flat, \hat \mu} (\Var_F)$ the quotient of the abelian monoid 
generated by isomorphism classes of  quasi-projective varieties over $F$
with a good $\hat \mu$-action by the standard cut and paste relations. We denote 
by $K_+^{\hat \mu} (\Var_F)$  the Grothendieck semi-ring of $F$-varieties with $\hat \mu$-action as considered in \cite{barc} and \cite{seattle}. It is the quotient of
$K_+^{\flat, \hat \mu} (\Var_F)$
by the following additional relations:  
for every quasi-projective $F$-variety $X$ with good $\hat \mu$-action, for every finite dimensional $F$-vector space $V$ endowed with two good linear actions
$\varrho$ and $\varrho'$, the class of $X \times (V, \varrho)$ is equal to the class of $X \times (V, \varrho')$.
We denote by $K^{\hat \mu} (\Var_F)$ the corresponding Grothendieck ring.

For any $s \in \Qq_{>0}$, let
$t_{s} \in F\llp t \rrp^{\rm alg}$ such that
$t_1 = t$ and
$t_{as} = t^a_s$ for any $s$ and any $a \geq 1$.
Set ${\bf{t}}_{k/m} = \rv (t_{k/m} ) \in V_{k/m}$.
Let $X$ be an  $F\llp t \rrp$-definable set in the generalized residue structure $\RES$.
Thus, for some $n \geq 0$, $X$ is an  $F\llp t \rrp$-definable subset of $\RV^n$
whose image in $\G^n$ under $\val_{\rv}$ is finite.
When the image is a single point, there exists
a positive integer  $m$ and integers   $k_i$, $1 \leq i \leq n$, such that $X$ is 
 an $F\llp t \rrp$-definable subset of 
$\prod_{1 \leq i \leq n} V_{k_i /m}$.
The $\hat \mu$-action on $X$ factors through a $\mu_m$-action.
The image $\Theta (X)$ of the set $X$ by the $F\llp t^{1/m}\rrp$-definable function $g(x_1, \ldots, x_n) =
(x_1/ {\bf{t}}_{k_1 /m}, \ldots,x_n/ {\bf{t}}_{k_n /m})$ is an $F$-definable subset of $\kk^n$
which is endowed with  a $\mu_m$-action coming from the one on $X$. 
In general, the set $X$  is a disjoint union of definable subsets $X_j$ of the previous type.
Since an $F$-definable subset of $\kk^n$
is nothing but a constructible subset of $\mathbb{A}^n_F$,
there is a unique morphism of  semi-rings
\begin{equation}\label{4.3.1}
\Theta \colon K_+  (\RES) \longrightarrow K_+^{\flat, \hat \mu} (\Var_F)
\end{equation}
such that, for every $F\llp t \rrp$-definable set $X$ in the structure $\RES$,
\begin{equation}
\Theta ([X]) = \sum_j [\Theta (X_j)].
\end{equation}
One derives from (\ref{4.3.1}) a ring morphism
\begin{equation}\label{4.3.2}
\Theta \colon \, \mathord{!K}(\RES) \longrightarrow \,  K^{\hat \mu} (\Var_F).
\end{equation}
We shall also consider the morphism
\begin{equation}\label{4.3.3}
\Theta_0 \colon \, \mathord{!K}(\RES) \longrightarrow \,  K (\Var_F)
\end{equation}
obtained by composing 
the morphism (\ref{4.3.2}) with the
morphism $K^{\hat \mu} (\Var_F) \to K (\Var_F)$ induced by forgetting the $\hat \mu$-action.

\begin{prop}\label{Theta}
The morphisms \textup{(\ref{4.3.1})} and \textup{(\ref{4.3.2})}  are isomorphisms.
\end{prop}
\begin{proof}  
Let us prove that   (\ref{4.3.1}) is   injective.
Let $X$ and $X'$ be respectively 
$F\llp t \rrp$-definable subsets of $\RV^n$ and $\RV^{n'}$ 
whose respective images in $\G^n$ and $\G^{n'}$ under $\val_{\rv}$ are a single point
and choose a positive integer $m$ such that 
the $\hat \mu$-action on $X$ and $X'$ factors through a $\mu_m$-action. 
Consider the corresponding
$F\llp t^{1/m}\rrp$-definable functions $g$ and $g'$.
Let $f$ be a $\hat \mu_m$-invariant isomorphism
between $g'(X')$ and $g(X)$.
Then $g \inv  \circ f \circ g': X' \to X$ is an $F\llp t^{1/m}\rrp$-definable bijection $X' \to X$, which moreover is invariant
under the Galois group of $F\llp t^{1/m}\rrp/ F\llp t \rrp$ hence is an $F\llp t \rrp$-definable bijection.
In general, when the  images of $X$ and $X'$ under $\val_{\rv}$  are only supposed to be finite, if
$\Theta ([X]) = \Theta ([X'])$, 
one can write $X$  and $X'$ as  disjoint union of definable subsets $X_j$ and $X'_j$ of the previous type, $1 \leq j \leq r$, such that for all $j$,
$\Theta (X_j) = \Theta (X'_j)$, and injectivity of (\ref{4.3.1}) follows.

For surjectivity, by induction on dimension, it is enough 
to  prove that, for $ m\geq 1$, if $V$ is an irreducible quasi-projective variety over $F$ endowed with
a $\mu_m$-action,  then there exists an $F\llp t \rrp$-definable set $W$ over $\RES$ such that
$\Theta (W)$ is a  dense subset of $V$. We may assume, by partitioning, that the kernel of the action is constant, so that the action is equivalent to an effective $\mu_{m'}$-action
for some $m' | m$, and for notational simplicity we take $m=m'$.
Set $U = V / \mu_m$. By Kummer theory there exists $f \in F (U)$ such that
$F (V) = F (U) (f^{1/m})$. Up to shrinking $V$, we may assume $f$ is regular and does not vanish on $U$. 
It follows that $V$ is isomorphic to the closed set
$V^\ast = \{(u, z) \in U \times \mathbb{G}_m ; f(u) = z^m\}$,
with $\mu_m$-action the trivial action on the $U$-factor and the standard one on the $\mathbb{G}_m$-factor.
If one sets 
$W= \{(u, z) \in U \times V_{1/m} ; f(u) =  t z^m\}$, one gets that $\Theta (W) = V$.

Since any linear $\mu_m$-action on $\mathbb{A}^n_F$ is diagonalizable, the relations involved in dropping the ``flat" from \eqref{4.3.1}
to \eqref{4.3.2} are just those implicit in adding the $!$ on the left hand side.  So  the bijectivity of
(\ref{4.3.2}) follows from the one of (\ref{4.3.1}).
\end{proof}

\section{\'Etale Euler characteristics with compact supports}\label{section:bb}

 \subsection{\'Etale cohomology with compact supports of semi-algebraic sets}

 Let $K$ be a complete non-archimedean normed field.
 Let $X$ be an algebraic variety over $K$ and write $X^{an}$ for its analytification in the sense of Berkovich.
 Assume now $X$ is affine.
 A semi-algebraic subset of $X^{an}$, in the sense of
  \cite{duc}, is a subset of $X^{an}$
  defined by a finite Boolean combination of inequalities
  $\vert f \vert \leq \lambda \vert g \vert$ with $f$ and $g$ regular functions on $X$ and
  $\lambda \in \mathbb{R}$.

 We denote by $\overline K$ the completion of a separable closure of $K$ and by
 $G$ the Galois group $\mathrm{Gal} (\overline K / K)$.
We set
 $\overline{X^{an}} = X^{an} \widehat{\otimes} \overline K$ and for
 $U$ a semi-algebraic subset of $X^{an}$ we denote by
 $\overline U$ the preimage of $U$ in $\overline{X^{an}}$ under the canonical morphism
 $\overline{X^{an}} \to X^{an}$.
 Let $\ell$ be a prime number different from the residue characteristic of $K$.

 Let $U$ be a locally closed semi-algebraic subset of
 $X^{an}$. For any finite torsion ring $R$, the theory of germs in \cite{berket}
 provides
 \'etale cohomology groups with compact supports
 $H^i_c (\overline U, R)$ which coincide with the ones defined there when $U$ is an affinoid domain of $X^{an}$. 
 These groups are also endowed with an action of the Galois group $G$.
 We shall set
 $H^i_c (\overline U, \mathbb{Q}_{\ell})
 =
 \mathbb{Q}_{\ell} \otimes_{\mathbb{Z}_{\ell}} \liminv H^i_c (\overline U, \mathbb{Z}/ \ell^n).
 $

 We shall use the following properties of the functor $U \mapsto H^i_c (\overline U, \mathbb{Q}_{\ell})$ which are proved by F. Martin in \cite{FM}:

 \begin{theo}\label{bb}Let $X$ be an affine algebraic variety over $K$ of dimension $d$.
 Let $U$ be a locally closed semi-algebraic subset of
 $X^{an}$.
 \begin{enumerate}
 \item[(1)]The groups $H^i_c (\overline U, \mathbb{Q}_{\ell})$ are finite dimension $\mathbb{Q}_{\ell}$-vector spaces, endowed with a $G$-action, and
 $H^i_c (\overline U, \mathbb{Q}_{\ell}) = 0$ for $i > 2d$.
  \item[(2)] If $V$ is  a semi-algebraic subset of $U$ which is  open in $U$ with complement  $F = U \setminus V$, there is a  long exact sequence
  \begin{equation}
  \longrightarrow H^{i - 1}_c (\overline F, \mathbb{Q}_{\ell})
  \longrightarrow  H^i_c (\overline V, \mathbb{Q}_{\ell}) \longrightarrow  H^i_c (\overline U, \mathbb{Q}_{\ell}) \longrightarrow H^i_c (\overline F, \mathbb{Q}_{\ell}) \longrightarrow
 \end{equation}
 \item[(3)]Let $Y$ be an affine algebraic variety over $K$ and let $V$ 
 be a locally closed semi-algebraic subset of
 $Y^{an}$.
 There are canonical K\"unneth isomorphisms
\begin{equation}
 \bigoplus_{i + j = n} H^i_c (\overline U, \mathbb{Q}_{\ell}) \otimes
 H^j_c (\overline V, \mathbb{Q}_{\ell})
 \simeq H^{n}_c (\overline{U \times V}, \mathbb{Q}_{\ell}).
\end{equation}
 \end{enumerate}
  \end{theo}

\begin{rema}
 We shall only make use  of Theorem \ref{bb} when $X = \mathbb{A}^n$ and
 $K = k \llp t \rrp$ with $k$ a field of characteristic zero. Note also that, 
 though in subsequent arXiv versions of \cite{FM} the proof of Theorem \ref{bb}  relies on
 Theorem 1.1 of \cite{berkv3} which uses de Jong's results on alterations and 
Gabber's weak uniformization
theorem, the first version 
is  based on Corollary 5.5 of \cite{berkv1}, which does not use any of these results.
 \end{rema}

 \subsection{Definition of $\etEU$}
 We denote by  $G\mathrm{-Mod}$
 the category of
 $\mathbb{Q}_{\ell} [G]$-modules that are
 finite dimensional as  $\mathbb{Q}_{\ell}$-vector spaces
 and by
 $ K (G\mathrm{-Mod})$ the corresponding Grothendieck ring.
 Let $K$ be a valued field endowed with a rank one valuation, that is,
 with $\Gamma (K) \subset \mathbb{R}$.
 We can consider the norm $\exp (- \val)$ on $K$.
 Let $U$ be  an $\ACVF_K$-definable subset of $\VF^n$.
 By quantifier elimination it is defined by a
 finite Boolean combination of inequalities
  $\val (f)  \geq \val (g) + \alpha$ where $f$ and $g$ are polynomials
  and $\alpha$ in
  $\Gamma (K) \otimes \mathbb{Q} $.
  Thus,
  after exponentiating, one attach canonically
  to $U$ the  semi-algebraic subset $U^{an}$
  of $(\mathbb{A}^n_{\widehat K})^{an}$ defined by the corresponding inequalities, with $\widehat K$ the completion of $K$, and also a semi-algebraic subset
  $\overline{U^{an}}$ of
  $\mathbb{A}^n_{\overline K}$.
   When $U^{an}$ is locally closed, we define
   $ \etEU (U)$ as the class of
\begin{equation}\sum_i (- 1)^i [H^i_c ( \overline{U^{an}}, \mathbb{Q}_{\ell})]
\end{equation}
   in
   $K (G\mathrm{-Mod})$. It  follows from (1) of Theorem \ref{bb}
   that this is well-defined.

\begin{lem}\label{part}Let $U$ be    an $\ACVF_K$-definable subset of $\VF^n$.
Then there exists a finite partition of $U$ into
$\ACVF_K$-definable subsets $U_i$ such that each $U_i^{an}$ is locally closed.
   \end{lem}

   \begin{proof}The set $U$ is the union
   of sets $U_i$ defined by conjunctions of formulas of the form $\val (f )< \val (g)$,
$f=0$, or $\val (f ) =  \val (g)$,
with $f$ and $g$ polynomials.  Since the
intersection and intersection of two locally closed sets are locally closed, it suffices to show that each of these basic
forms
are  locally closed.
Since $|f|$  and $|g|$
are continuous functions for the  Berkovich topology with values in
$\Rr_{\geq 0}$,
the sets defined by  $ f=0$ and  $\val (f) = \val (g)$ are closed, as well as $\val (f) \leq
\val (g)$. The remaining
kind of set, $\val (f )< \val (g)$, is the difference between  $\val (f )\leq \val
(g)$ and $\val (f) = \val (g)$, hence is locally
closed.
 \end{proof}

   \begin{prop}There exists a unique ring morphism
   \begin{equation}
   \etEU \colon
 K (\VF) \longrightarrow K (G\mathrm{-Mod})
 \end{equation}
 such that $ \etEU ([U]) = \etEU (U)$ when $U$  is an $\ACVF_K$-definable subset of $\VF^n$ such that $U^{an}$ is locally closed.
   \end{prop}

   \begin{proof}
   Let $U$ be    an $\ACVF_K$-definable subset of $\VF^n$.
   Choose a partition of   $U$ into
$\ACVF_K$-definable subsets $U_i$, $1 \leq i \leq r$, such that each $U_i^{an}$ is locally closed.
If $U^{an}$ is locally closed,
it follows from (2) in Theorem \ref{bb}, using induction on $r$, that $\etEU (U) = \sum_i \etEU (U_i)$.
For general $U$, set
   $ \etEU (U) = \sum_i \etEU (U_i)$.
 This is independent of the choice of the partition $U_i$.
   Indeed, if $U'_j$ is a finer such partition with
   $(U'_j)^{an}$  locally closed, then
   $\sum_i \etEU (U_i) = \sum_j \etEU (U'_j)$
   by
   the previous remark, and two
   such partitions always have a common refinement.
   Note that $ \etEU (U)$ depends only on the isomorphism class
   of $U$ as a definable set.
   Indeed, when $f$ is a polynomial isomorphism $f \colon U \to U'$ (with inverse given by a polynomial function),
   and $U^{an}$ and $(U')^{an}$ are locally closed, this is clear by
   functoriality of $H^{\bullet}_c$, and in general one can reduce to this case by taking suitable
   partitions
   $U_i$ and $U'_i$ of $U$ and $U'$.    Thus, if now $U$ is an $\ACVF_K$-definable subset of $X$, with $X$ an affine variety over $K$,
   if $i$  is some embedding of $X$ in an affine space $\mathbb{A}^n$,  $\etEU (i(U))$ will not depend on $i$, so
   one may set $\etEU (U) = \etEU (i(U))$.  Note that,  by definition, $ K (\VF)$ is generated by classes of $\ACVF_K$-definable subsets of affine algebraic varieties over $K$.
Furthermore, by (2) in Theorem \ref{bb}, $ \etEU $ satisfies the additivity relation, thus existence and uniqueness
   of an additive map 
  $ \etEU \colon
 K (\VF) \rightarrow K (G\mathrm{-Mod})$
 with the required property 
 follows. Its multiplicativity is a consequence from Property
      (3) in Theorem \ref{bb}.
   \end{proof}

 \subsection{Definition of $\eu_{\text{\'et}}$}
 We now assume for the rest of this section  that $K = F \llp t \rrp$ with $F$ algebraically closed of characteristic zero.
Thus the Galois group $G$  may be identified with $\hat \mu$ in the standard way, namely to an element
$\zeta = (\zeta_n)_{n \geq 1} \in \hat \mu$ corresponds the unique element $\sigma \in G$
such that, for any $n \geq 1$, $\sigma (x) = \zeta_n x$ if $x^n = t$.

  Let $X$ be an $F$-variety  endowed with a $\hat \mu$-action factoring for some some $n$ through a $\mu_n$-action.
The $\ell$-adic \'etale cohomology groups
$H^i_c (X, \mathbb{Q}_{\ell})$
are endowed with a $\hat \mu$-action, and we may consider the element
\begin{equation}\eu_{\text{\'et}} (X) := \sum_i  (-1)^i [H^i_c (X, \mathbb{Q}_{\ell})]
\end{equation}
in $K (\hat \mu\mathrm{-Mod})$.
Note that $\eu_{\text{\'et}} ([V, \varrho]) = 1$ for any finite dimensional $F$-vector space $V$
endowed with a $\hat \mu$-action factoring for some $n$ through a linear $\mu_n$-action.
Thus, $\eu_{\text{\'et}}$ factors to give rise to
a morphism
\begin{equation}
\eu_{\text{\'et}} \colon  K^{\hat \mu} (\Var_F) \longrightarrow \, K (\hat \mu\mathrm{-Mod}).
\end{equation}
Furthermore,
the morphism
$\eu_{\text{\'et}} \circ  \, \Theta$, with $\Theta$ as in (\ref{4.3.2}),
 factors through
$\mathord{!K}(\RES) / ( [\mathbb{A}^1] - 1)$
and gives rise to a morphism
 \begin{equation}
 \eu_{\text{\'et}} \colon \mathord{!K}(\RES) / ( [\mathbb{A}^1] - 1)
 \longrightarrow K (\hat \mu\mathrm{-Mod}).
 \end{equation}

 \subsection{Compatibility}\label{5.4}
 We have the following fundamental compatibility property between
 $\etEU$
 and
  $\eu_{\text{\'et}}$.

 \begin{theo}\label{theocompat}The diagram
\begin{equation}\label{compatibility}\xymatrix{
K (\VF) \ar[rr]^{\GEU} \ar[dr]_{ \etEU}& &\mathord{!K}(\RES) / ( [\mathbb{A}^1] - 1) \ar[dl]^{\eu_{\text{\'et}}}\\
&K (\hat \mu\mathrm{-Mod})& }
\end{equation}
is commutative.
\end{theo}

\begin{proof}It is enough to prove that if  $X$ is a definable subset of
$\VF^n$, then
$\etEU (X) = \eu_{\text{\'et}} (\GEU ([X]))$.
Using the notations of (\ref{bbL}) and the isomorphism
(\ref{deco}), we may assume
the class of $X$ in $K_+ (\VF)$ is of the form
$ \mathbb{L} (\Psi (a \otimes b))$
with
$a$
in
$K_+ (\RES [m])$
and
$b$ in $K_+ (\Gamma [r])$.

If $r \geq 1$,
$\GEU ([X]) = 0$ by construction of $\GEU$.
Indeed, with the notations from \ref{2.5}, $\chi_r ([a \otimes b]) = \chi (b) \cdot [\mathbb{G}_m]^{r} \cdot a$,
which implies that $E_n (a \otimes b) = 0$ for $n \geq r$, and
$\GEU ([X]) = 0$ follows.
To prove that $\etEU (X) = 0$,
it is enough by multiplicativity of $\etEU$,  to prove that
$\etEU (\mathbb{L} (\Psi (1 \otimes b))) = 0$, which follows from Lemma \ref{chi0}.

Thus, we may assume $r = 0$ and $[X] =  \mathbb{L} (\Psi ([Z] \otimes 1))$,
with $Z$ a definable subset in $\RES [n]$.
Since $\GEU (\mathbb{L} (Z))$ is equal to the class of $Z$ in
 $\mathord{!K}(\RES) / ( [\mathbb{A}^1] - 1)$, it is enough to prove
that 
$ \etEU([\mathbb{L} (Z)]) = \eu_{\text{\'et}} ([Z])$.

Let $Z$ be a definable subset of $\RES [n]$.
For some integer $d$, $Z$ is a definable subset of $\RES^d$ and, after partitioning $Z$ into a finite number of definable
sets, we may assume that, with the notation of \S \ref{3.4},
there exists a positive integer  $m$ and
integers
 $k_i$, $1 \leq i \leq d$,
such that $Z$ is a definable subset of
$\prod_{1 \leq i \leq d} V_{k_i /m}$. Consider the  $F\llp t^{1/m}\rrp$-definable  isomorphism
$g \colon \prod_{1 \leq i \leq d} V_{k_i /m} \to \mathbb{A}_F^d$
given
by $g (x_1, \ldots, x_d) =
(x_1/ {\bf{t}}_{k_1/m}, \ldots,x_d/ {\bf{t}}_{k_d/m})$. The Galois action on the space $\prod_{1 \leq i \leq d} V_{k_i /m}$ factorizes through a $\mu_m$-action
  and $g$ becomes $\mu_m$-equivariant if one endows
 $\mathbb{A}_F^d$ with the action $\varrho$ of $\mu_m$ given by 
  $\zeta \cdot (y_1, \ldots, y_d) = (\zeta^{k_1} y_1, \ldots, \zeta^{k_d} y_d)$.  
 The set $Y = g (Z)$ is an $F$-definable subset of $\mathbb{A}_F^d$. We shall still denote by
 $g$ the induced $F\llp t^{1/m}\rrp$-definable  isomorphism 
 $g \colon Z\to Y$ and by $\varrho$ the induced action of $\mu_m$ on $Y$.
 We may assume $Z$ is of $\RV$-dimension $n$. Indeed, if $Z$ is of $\RV$-dimension  $< n$,
 there exists a definable morphism
 $h \colon Z \to \RV^{n-1}$ with finite fibers. Let $i \colon  \RV^{n-1} \to \RV^n$ denote the canonical inclusion and set $f = i \circ h$.
 Since $\mathbb{L} ((Z, f))=
\mathbb{L} ((Z, h))\times \Mm$,
with $\Mm$ the maximal ideal
and 
$\etEU ([ \Mm]) = 1$, it follows that
$\etEU ([ \mathbb{L} ((Z, f))]) =
\etEU ([ \mathbb{L} ((Z, h))]) $
by multiplicativity and we may conclude by induction on $n$ in this case.
Thus, by additivity, we may assume $Y$ is a smooth variety over $F$ of pure dimension $n$
and that the morphism  $f _Y \colon Y \to \mathbb{A}_F^n$ given  by projection to the first $n$ factors
has finite fibers. It follows that the morphism $f_Z \colon Z \to \prod_{1 \leq i \leq n} V_{k_i /m}$ given by projecting to the first $n$ factors
has finite fibers too.
 The definable subset
$ \mathbb{L} ((Y, f_Y))$  of  $ (\VF^{\times})^n \times \mathbb{A}_F^d$ is the isomorphic image of  the subset $\mathbb{L} ((Z, f_Z))$ of $(\VF^{\times})^n \times\prod_{1 \leq i \leq d} V_{k_i /m}$ under the mapping
$\tilde g \colon (X_1, \ldots, X_n, x_1, \ldots, x_d) \mapsto (X_1/ t_{k_1/m}, \ldots,X _n/ t_{k_n/m}, x_1/ {\bf{t}}_{k_1/m}, \ldots,x_d/ {\bf{t}}_{k_d/m})$.
It is endowed with a $\mu_m$-action $\tilde \varrho$ given by
$\zeta \cdot (X_1, \ldots, X_n,y_1, \ldots, y_d) = (\zeta^{k_1} X_1, \ldots, \zeta^{k_n} X_n, \zeta^{k_1} y_1, \ldots, \zeta^{k_d} y_d)$.

Let us consider  the formal completion $\mathcal{Y}$ of $Y \otimes F \llb t \rrb$. Denote by
 $\mathcal{Y}_{\eta}$ the analytic  generic fiber of $\mathcal{Y}$ and by $\pi$ the 
reduction map $\pi \colon \mathcal{Y}_{\eta} \to Y$. 
The $\mu_m$-action $\varrho$ induces an action which we still denote by $\varrho$ on $\mathcal{Y}$  and $\mathcal{Y}_{\eta}$.
By Lemma  13.2 in \cite{HK} and its proof, $\mathbb{L}  (Y)^{an}$ is isomorphic to $\pi^{-1} (Y)$.
Furthermore, under this isomorphism the 
$\mu_m$-action on $\mathbb{L}  (Y)^{an}$ induced from $\tilde \varrho$ corresponds to the action $\varrho$ on 
$\pi^{-1} (Y)$. 
Denote by $\pi_m$ the projection $\hat \mu \to \mu_m$. The mapping $\tilde g$ induces an isomorphism between the spaces $\overline{\mathbb{L}  (Z)^{an}}$ and $\overline{\mathbb{L}  (Y)^{an}}$
under which
the Galois action on 
$\overline{\mathbb{L}  (Z)^{an}}$ corresponds to the Galois action twisted by $\tilde \varrho$ on
$\overline{\mathbb{L}  (Y)^{an}}$, namely the action for which  an element $\sigma$ of $\hat \mu$ acts on 
$\overline{\mathbb{L}  (Y)^{an}}$ by $y \mapsto \sigma \cdot \tilde \varrho (\pi_m (\sigma)) \cdot y =  \varrho (\pi_m (\sigma)) \cdot \sigma \cdot y$.
It follows that, for
$i \geq 0$, 
$H^i_c (\overline{\mathbb{L}  (Z)^{an}}, \mathbb{Q}_{\ell})$ is isomorphic to $H^i_c (\overline{\pi^{-1} (Y)}, \mathbb{Q}_{\ell})$ and that,
since the Galois action on $H^i_c (\overline{\pi^{-1} (Y)}, \mathbb{Q}_{\ell})$
is trivial, cf. Lemma \ref{redu},  that 
the Galois action on 
$H^i_c (\overline{\mathbb{L}  (Z)^{an}}, \mathbb{Q}_{\ell})$ factorizes through $\mu_m$ and corresponds to the action induced by $\varrho$
on $H^i_c (\overline{\pi^{-1} (Y)}, \mathbb{Q}_{\ell})$.
By 
Lemma \ref{redu} 
there is a canonical isomorphism, equivariant for the action $\varrho$,
\begin{equation}\label{5.4.3}
H^i_c (\overline{\pi^{-1} (Y)}, \mathbb{Q}_{\ell}) \simeq H^{2n - i}_c (Y,  \mathbb{Q}_{\ell} (n))^{\vee},
\end{equation}
with the subscript ${}^{\vee}$ standing for the dual of a $\mathbb{Q}_{\ell}$-vector space.
Since
$ \etEU([\mathbb{L} (Z)])$ is equal to $\sum_i (-1)^i [H^i_c (\overline{\mathbb{L}  (Z)^{an}}, \mathbb{Q}_{\ell})]$,
it follows from (\ref{5.4.3}) that it is equal to
  $\sum_i (-1)^i [H^{2n -i}_c (Y,  \mathbb{Q}_{\ell} (n))^{\vee}] = \sum_i (-1)^i [H^i_c (Y,  \mathbb{Q}_{\ell} (n))^{\vee}]$, with the $\mu_m$-action induced from 
$\varrho$. 
Let us  note that a finite dimensional vector space $\mathbb{Q}_{\ell}$-vector space $V$ with $\mu_m$-action 
has the same class in $K (\hat \mu\mathrm{-Mod})$ as its dual $V^{\vee}$ endowed with the dual action and that, for any integer $n$,
$V$ and the Tate twist $V(n)$ have the same class in $K (\hat \mu\mathrm{-Mod})$.
It follows  that $\etEU([\mathbb{L} (Z)]) = \sum_i (-1)^i [H^i_c (Y,  \mathbb{Q}_{\ell})]$, with action on the right-hand side induced from 
$\varrho$, hence 
$\etEU([\mathbb{L} (Z)]) = \eu_{\text{\'et}} ([Z])$.
\end{proof}

\begin{lem}\label{chi0}Let  $m \geq 1$ be an integer and let $Z$ be a definable subset  of
$\G^{m}$. Then we have
\begin{equation}\etEU (\val^{-1}(Z)) = 0.
\end{equation}
\end{lem}

\begin{proof}By quantifier elimination and cell decomposition in o-minimal structures, cf. \cite{vddries-tame},
using additivity  of $\etEU$, 
we may assume
there exists a definable subset $Z'$ of
$\G^{m- 1}$,
affine linear forms with rational coefficients $L_1$ and $L_2$ in variables $u_1$, \dots, $u_{m - 1}$
such that
$Z$ is defined by the conditions
$(u_1, \dots, u_{m - 1}) \in Z'$ and
\begin{equation}\label{cond}L_1 (u_1, \dots, u_{m- 1}) \,  \square_{1} \,  u_m \,
\square_{2} \, L_2 (u_1, \dots, u_{m - 1}),
\end{equation}
where $\square_{ 1}$ and $\square_{2}$ are of one of the following four types:
\begin{enumerate}
\item[(1)] no condition,
\item[(2)] $=$ and no condition,
\item[(3)]  $<$ and $<$,
\item[(4)]  $<$ and no condition,
\item[(5)] no condition and $<$.
\end{enumerate}
Thus $\val^{-1}(Z)$
is the set defined by the conditions $(x_1, \dots, x_{m - 1}) \in \val^{-1}(Z')$ and
\begin{equation}L_1 (\val(x_1), \dots, \val(x_{m - 1})) \,  \square_{1} \,  \val (x_m) \,
\square_{2} \, L_2 (\val(x_1), \dots, \val(x_{m - 1})).
\end{equation}
In case (1), $\val^{-1}(Z)$ is equal to the product of $\val^{-1}(Z')$ by the open annulus $C = \VF \setminus \{0\}$ and we deduce
$\etEU (\val^{-1}(Z)) = 0$ from the fact that $\etEU (C)  = 0$.
In case (2), 
$Z$ is defined by the conditions
$(u_1, \dots, u_{m - 1}) \in Z'$ and $u_m = \sum_{1\leq i < m} a_i u_i + b$ with
$a_i \in \mathbb{Z}$, $ 1 \leq  i <m$, and $b$ in $\mathbb{Q}$. We may rewrite the last condition
in the form $\sum_{1\leq i \leq m} b_i u_i = c$ with $(b_1, \dots, b_m)$ a primitive vector in $\G^m$ and $ c \in \mathbb{Q}$.
Thus, up to to changing the coordinates in $\G^m$, we may assume $Z$ is defined
by $(u_1, \dots, u_{m - 1}) \in Z'$ and $u_m = c$,
so that  $\val^{-1}(Z)$ is equal to the product of $\val^{-1}(Z')$ by the closed annulus $D$ defined by $\val (x_m) = c$.
Since  $\etEU (D) = 0$ by a classical computation, we get that $\etEU (\val^{-1}(Z)) = 0$ in this case.
To deal with the remaining cases, consider the sets
$Z_1$ and $Z_2$ defined respectively by 
$(u_1, \dots, u_{m - 1}) \in Z'$ and $u_m \leq L_1(u_1, \dots, u_{m- 1})$, resp. $u_m \geq L_2 (u_1, \dots, u_{m- 1})$.
It is enough to prove that $\etEU (\val^{-1}(Z_1)) = 0$ and $\etEU (\val^{-1}(Z_2)) = 0$, since then, by additivity, the result will follow  from case (1).
Let us prove $\etEU (\val^{-1}(Z_2)) = 0$.
Similarly as in case (2), after a change of variable one may assume $Z_2$ is defined by
$(u_1, \dots, u_{m - 1}) \in Z'$ and $u_m \geq c$, for some rational number $c$, so that
$\val^{-1}(Z_2)$ is equal to the product of $\val^{-1}(Z')$ by the set $E$ defined $\val (x_m) \geq c$ (the complement of the origin in the closed ball of valuative radius $c$).
Since $\etEU (E) = 0$, we deduce that 
$\etEU (\val^{-1}(Z_2)) = 0$.
The proof that  $\etEU (\val^{-1}(Z_1)) = 0$ is similar. 
\end{proof}

\begin{lem}\label{redu}
Let $\mathcal{X}$ be a smooth formal scheme  of finite type over the valuation ring of $K$ with special fiber $X$ of pure dimension $n$ and
analytic generic fiber $\mathcal{X}_{\eta}$. Let $\pi \colon \mathcal{X}_{\eta} \to X$
be the reduction map.
Let $S$ be a smooth closed subscheme of $X$.
Then there exist canonical isomorphisms
\begin{equation}\label{5.4.7}
H^i_c (\overline{\pi^{-1} (S)}, \mathbb{Q}_{\ell}) \simeq H^{2n - i}_c (S,  \mathbb{Q}_{\ell} (n))^{\vee},
\end{equation}
with ${}^\vee$ standing for the dual vector space. In particular, the Galois action on $H^i_c (\overline{\pi^{-1} (S)}, \mathbb{Q}_{\ell})$ is trivial for $i \geq 0$.
Assume furthermore that a finite group scheme $H$ acts on $\mathcal{X}$ inducing an action on $X$ such that $S$ is globally invariant by $H$.
Then the isomorphism \textup{(\ref{5.4.7})} is equivariant for the $H$-action induced on both sides.
\end{lem}

\begin{proof}By Corollary 2.5 of \cite{berkv2}, for any finite torsion group $\Lambda$, we have a canonical isomorphism
\begin{equation}\label{cor2.5}
R \Gamma_c (\overline{\pi^{-1} (S)}, \Lambda_{\mathcal{X}_{\eta}}) \simeq R \Gamma_{S} (X, R \psi_{\eta} (\Lambda_{\mathcal{X}_{\eta}})).
\end{equation}
One checks  by inspection of the proof in \cite{berkv2} that this
isomorphism  is $H$-equivariant.
By triviality of vanishing cycles for smooth analytic spaces, cf. Corollary 5.7 of \cite{berkv1},
$R^q \psi_{\eta} (\Lambda_{\mathcal{X}_{\eta}}) = 0$ for $q >0$ and $R^0 \psi_{\eta} (\Lambda_{\mathcal{X}_{\eta}}) = \Lambda_{X}$,
hence it follows that there  exist canonical $H$-equivariant  isomorphisms
\begin{equation}
H^i_c (\overline{\pi^{-1} (S)},\Lambda) \simeq H^i_S  (X, \Lambda).
\end{equation}
We may assume $S$ is of pure codimension $r$, hence, by purity, we have canonical $H$-equivariant 
isomorphisms $H^i_S  (X, \Lambda) \simeq H^{i - 2r} (S, \Lambda (-r))$,
so we get canonical  $H$-equivariant isomorphisms
\begin{equation}
H^i_c (\overline{\pi^{-1} (S)}, \Lambda) \simeq H^{i - 2r}  (S, \Lambda(-r)).
\end{equation}
Note that $S$ is smooth of dimension $d =n - r$.
Thus, for $j = 2d - (i - 2r) = 2n - i$,
the canonical morphism
\begin{equation}
H^j_c  (S, \Lambda (d + r)) \times H^{i - 2r}  (S, \Lambda(-r)) \longrightarrow H^{2d}_c  (S, \Lambda (d))  \simeq \Lambda
\end{equation}
is a perfect pairing of finite groups by Poincar\'e Duality.
The statement  follows by passing to the limit over torsion coefficients $\mathbb{Z}/  \ell^m  \mathbb{Z}$.
\end{proof}

\subsection{A fixed point formula}

The following version of the Lefschetz fixed point Theorem is classical and follows in particular from
Theorem 3.2 of \cite{DeligneLusztig}:

\begin{prop}\label{lfp}
Let $Y$ be a quasi-projective variety over an algebraically closed field of characteristic zero.
Let $T$ be a finite order automorphism of $X$.
Let $Y^T$ be the fixed point set of $T$ and denote 
by
$\chi_{c}(Y^T, \mathbb{Q}_{\ell})$ its $\ell$-adic Euler characteristic with compact supports.
Then
\begin{equation}
\chi_{c} (Y^T, \mathbb{Q}_{\ell})
=
\tr (T; H^{\bullet}_c (Y, \mathbb{Q}_{\ell})).
\end{equation}
\end{prop}

Let us  denote by $\Theta_0$ the morphism
\begin{equation}\label{5.5.2}
\Theta_0 \colon \mathord{!K}(\RES) / ([\mathbb{A}^1] - 1) \longrightarrow   K(\Var_F) /([\mathbb{A}^1] - 1)
\end{equation}
induced by the morphism $\Theta_0$ of (\ref{4.3.3}).
Denote by $\chi_c$ the morphism
$ K(\Var_F) / ([\mathbb{A}^1] - 1)    \rightarrow \mathbb{Z}$
 induced by  the $\ell$-adic Euler characteristic with compact supports.

Combining Theorem \ref{theocompat} with Proposition \ref{lfp} we obtain the following
fixed point formula:

\begin{theo}\label{gentrace}
Let $X$ be an $\ACVF_K$-definable subset of $\VF^n$.
Then, for every $m \geq 1$,
\begin{equation}
\tr (\varphi^m;  \etEU ([X]))
=
\chi_c (\Theta_0 \circ \EU_{\Gamma, m} ([X])).
\end{equation}
\end{theo}

\begin{proof}Let $m \geq 1$. By Theorem \ref{theocompat},
\begin{equation}
\tr (\varphi^m;  \etEU ([X]))
=
\tr (\varphi^{m};   \eu_{\text{\'et}} (\GEU ([X])))
\end{equation}
and by
Proposition \ref{lfp},
\begin{equation}
\tr (\varphi^{m};   \eu_{\text{\'et}} (\GEU ([X])))
=
\chi_c (\Theta_0 \circ \EU_{\Gamma, m} ([X])).
\end{equation}
The result follows.
\end{proof}

\section{Proof of Theorem \ref{mt}}

In this section  we are working over $F \llp t \rrp$, with $F = \mathbb{C}$. Our aim is to prove Theorem \ref{mt},
namely that, for every $m \geq 1$, with the notation from the Introduction,
\begin{equation*}
\chi_c (\mathcal{X}_{m, x}) =
\Lambda (M_x^m).
\end{equation*}

\subsection{Using comparison results}Let $X$ be a smooth complex variety and $f $ be a regular function on $X$.
Let $x$ be a closed point of the fiber $f^{-1}(0)$.
We shall use the notation introduced in
Corollary \ref{4}.
Thus $\pi$ denotes the reduction map $X (\Oo) \to X (\kk)$, and we consider the $\ACVF_{F\llp t \rrp}$-definable sets
\begin{equation}
X_{t, x}= \{y \in X(\Oo); f(y)=t  \,\, \mathrm{and} \,\, \pi (y) = x \}\end{equation} and
\begin{equation}\label{defX}
\mathcal{X}_x  = \{y \in X(\Oo); \rv f(y) = \rv(t) \,\, \mathrm{and} \,\, \pi (y) = x \}.
\end{equation}

The definable set  $X_{t, x}$ is closely related to   the analytic Milnor fiber 
$\mathcal{F}_x$ introduced in  \S 9.1 of \cite{NS} whose definition we now recall.
Let $X_{\infty}$ be  the $t$-adic completion of $f \colon X \to \Spec \mathbb{C} [t]$ and let $X_{\eta}$ be its
its generic fiber (in the category of rigid
$F \llp t \rrp$-varieties). There is a canonical specialization morphism
$\mathrm{sp} \colon X_{\eta} \to X_{\infty}$ (cf. \S 2.2 of \cite{NS}) 
and  $\mathcal{F}_x $ is defined as $\mathrm{sp}^{-1} (x)$. It is an open rigid subspace of
$X_{\eta}$. It follows directly from the definitions that
$X_{t, x}^{an}$ and $\mathcal{F}_x^{an}$ may be canonically identified.

Fix a prime number $\ell$ and denote by $\varphi$ the  topological generator
of $\hat \mu (\mathbb{C}) = \mathrm{Gal} (\mathbb{C} \llp t \rrp^{\rm alg} / \mathbb{C} \llp t \rrp)$ given by  the family
$(\zeta_n)_{n \geq 1}$ with $\zeta_n = \exp (2 i \pi /n)$.
It follows from Theorem 9.2 from \cite{NS} (more precisely, from its proof ; note that in the notation of loc. cit. the exponent $an$ is omitted), which is a consequence of the second isomorphism proved in
 \cite{berkv2} Corollary 3.5,
that there exist isomorphisms
\begin{equation}
H^i (F_x, \mathbb{Q}) \otimes_{\Qq} \mathbb{Q}_{\ell} \simeq H^i (\mathcal{F}_x^{an} \widehat{\otimes} \, \widehat{\mathbb{C} \llp t \rrp^{\rm alg}}, \mathbb{Q}_{\ell} )
\end{equation}
compatible with the action of $M_x$ and $\varphi$. Here  $F_x$ is the topological Milnor fiber defined in (\ref{mfdef}).
It follows that, for every $m \geq 0$,
\begin{equation}
\Lambda (M_x^m) = \tr (\varphi^m;
H^{\bullet} (\mathcal{F}_x^{an} \widehat{\otimes} \, \widehat{\mathbb{C} \llp t \rrp^{\rm alg}}, \mathbb{Q}_{\ell})).
\end{equation}
By Poincar\'e Duality as established in \S 7.3 of \cite{berket},
there is a perfect duality
\[
H^{i} (\mathcal{F}_x^{an} \widehat{\otimes} \,  \widehat{\mathbb{C} \llp t \rrp^{\rm alg}}, \mathbb{Z}/  \ell^n\mathbb{Z}))
\times
H^{2d - i}_c (\mathcal{F}_x^{an} \widehat{\otimes} \, \widehat{\mathbb{C} \llp t \rrp^{\rm alg}}, \mathbb{Z}/  \ell^n\mathbb{Z} (d)) \to
\mathbb{Z}/  \ell^n\mathbb{Z},
\]
with $d$ the dimension of $X$, which is compatible with the $\varphi$-action.
Hence, after  taking the limit over $n$ and tensoring with
$\mathbb{Q}_{\ell}$,  one deduces
that, for every $m \geq 0$,
\begin{equation}\Lambda (M_x^m) = \tr (\varphi^{m};
H^{\bullet}_c (\mathcal{F}_x^{an} \widehat{\otimes} \,  \widehat{\mathbb{C} \llp t \rrp^{\rm alg}}, \mathbb{Q}_{\ell})),
\end{equation}
which may be rewritten as
\begin{equation}\label{comp}
\Lambda (M_x^m) = \tr (\varphi^{m};
H^{\bullet}_c (X_{t, x}^{an} \widehat{\otimes} \,  \widehat{\mathbb{C} \llp t \rrp^{\rm alg}}, \mathbb{Q}_{\ell})).
\end{equation}

\begin{rema}With the notations of Corollary 3.5 of \cite{berkv2}, when $\mathcal{Y}$ is proper, it is explained in  Remark 3.8 (i) of  \cite{berkv2} how to deduce
the 
first isomorphism of Corollary
3.5 of \cite{berkv2}
directly from Theorem 5.1 in \cite{berkv1}  in the way indicated in \cite{fa}. When furthermore $\mathcal{X}_{\eta}$ is smooth (keeping the notations of loc. cit.),
the second isomorphism of Corollary
3.5 of \cite{berkv2} follows from the first by 
Poincar\'e Duality and Corollary 5.3.7 of \cite{berket}. In particular, for the use which is made of Corollary
3.5 of \cite{berkv2} in this paper, one may completely avoid using de Jong's results on stable reduction and one may rely only on results from  \cite{berket} and \cite{berkv1}.
\end{rema}

\subsection{Proof of Theorem \ref{mt}}\label{pft}Let $m \geq 1$.
With the previous notations, one may rewrite (\ref{comp})  as 
\begin{equation}\label{step1}
\Lambda (M_x^m) = \tr (\varphi^{m}; \etEU ([X_{t, x}])).
\end{equation}
On the other hand, by Theorem \ref{gentrace} we have
\begin{equation}\label{step3}
 \tr (\varphi^{m}; \etEU ([X_{t, x}]))
 =
 \chi_c (\Theta_0 (\EU_{\Gamma, m} ([X_{t, x})])).
\end{equation}
In Corollary \ref{4}, it is proven that
$\EU_{\Gamma, m}(X_{t, x}) =\mathcal{X}_x[m]$ as classes in $\mathord{!K}(\RES) / ([\mathbb{A}^1] - 1)$.
In particular, 
\begin{equation}
 \chi_c (\Theta_0 (\EU_{\Gamma, m} ([X_{t, x})]))
 =
 \chi_c (\mathcal{X}_{x} [m]).
\end{equation}
To conclude the proof it is thus enough to check  that
\begin{equation}
 \chi_c (\mathcal{X}_{x} [m])
 =\chi_c (\mathcal{X}_{m, x}).
\end{equation}
This may be  seen as follows. For $m \geq 1$,
\begin{equation}
\mathcal{X}_{m, x} =
\{\varphi \in X (\mathbb{C} [t] / t^{m + 1}); f (\varphi) = t^m \mod t^{m + 1},  \, \varphi (0) = x\}
\end{equation}
may be rewritten as
\begin{equation}
\{\varphi \in X (\mathbb{C} [t^{1/m}] / t^{(m+1)/m} );   f (\varphi)
= t  \mod t^{(m+ 1)/m},  \, \varphi (0) = x\}
\end{equation}
 or as
\begin{equation} \{\varphi \in X (\mathbb{C} [t^{1/m}] / t^{(m+1)/m}) ; \rv( f(\varphi)) = \rv(t),  \, \varphi (0) = x\}.
\end{equation}
Thus $\Theta (\mathcal{X}_{x} [m])$ and $\mathcal{X}_{m, x}$ have the same class in
$K^{\hat \mu} (\Var_F) / ([\mathbb{A}^1] - 1)$. The equality
 $\chi_c (\mathcal{X}_{x} [m])
 =\chi_c (\mathcal{X}_{m, x})$  follows. \qed

\section{Trace formulas and the motivic Serre invariant}\label{secserre}
\subsection{}In this section $F$ denotes a field of characteristic zero,
$K = F \llp t \rrp$, $K_m = F \llp t^{1/m}\rrp$ and $\bar K = \cup_{m \geq 1} K_m$. 
If $X$ is an $\ACVF_K$-definable set or an algebraic variety over $K$, we write
$X (m)$ and $\bar X$ for the objects obtained by extension of scalars to
$K_m $ and $\bar K$, respectively.
As in  (\ref{5.5.2}) we denote by $\Theta_0$ the morphism
\begin{equation}
\Theta_0 \colon \mathord{!K}(\RES) / ([\mathbb{A}^1] - 1) \longrightarrow   K(\Var_F) /([\mathbb{A}^1] - 1)
\end{equation}
induced by the morphism $\Theta_0$ of (\ref{4.3.3}).

\subsection{The motivic Serre invariant}
Let $R$ be a complete discrete valuation ring, with perfect residue field $F$ and field of fractions $K$.
We denote by $R^{sh}$ a strict Henselization of $R$ and by  $K^{sh}$ its field of fractions.
Let $X$ be a smooth quasi-compact rigid $K$-variety.
In \cite{LS}, using motivic integration on formal schemes,
for any such $X$ a canonical class $S (X) \in K(\Var_F) /([\mathbb{A}^1] - 1) $
is constructed, called  the motivic Serre invariant of $X$.
If $X$ is a smooth proper algebraic variety over $K$, one sets $S (X) = S (X^{rig})$,
with $X^{rig}$ the rigid analytification of $X$.

We have the following comparison between the morphism $\GEU$ and the motivic Serre invariant
in residue characteristic zero via the morphism $\Theta_0$:

\begin{prop}\label{compserre}Let $K = F \llp t \rrp$ with $F$ a field of characteristic zero.
Let $X$ be a smooth proper algebraic  variety over $K$.
Then, for every $m \geq 1$,
\begin{equation}
\Theta_0 (\EU_{\Gamma, m} ([X])) = S (X (m)).
\end{equation}
\end{prop}

\begin{proof}After replacing $F \llp t \rrp$ by $F \llp t^{1/m}\rrp$ we may assume $m = 1$.
Let $\mathcal{X}$ be a weak N\'eron model of $X$, cf. section 2.7 of \cite{LS}.
This means that $\mathcal{X}$
is a smooth $R$-variety endowed with an isomorphism $\mathcal{X}_K \to X$
such that the natural map
$\mathcal{X}(R^{sh}) \to X(K^{sh})$ is a bijection.
Consider the unique
definable subset $X_1$ of $X$ such that
for any valued field extension  $K'$ of $K$, with valuation ring $R'$,
$X_1 (K')$ is the image of
$\mathcal{X} (R')$ under the canonical
mapping $\mathcal{X} (R') \to X (K')$ (in fact $\mathcal{X}$ 
gives rise to a definable set and $X_1$ is its image through the natural map
$\mathcal{X} \to X$).
Let $X_{\not= 1}$ be the complement of $X_1$ in $X$.
By the very construction of  $\EU_{\Gamma, 1}$ and $S (X)$,
$\Theta_0 (\EU_{\Gamma, 1} ([X_1])) = S (X)$.
Thus it is enough to prove that
$\EU_{\Gamma, 1} ([X_{\not= 1}]) = 0$.
Since $X_{\not= 1} (F' \llp t \rrp) = \varnothing$
for every field extension $F'$ of $F$ by the N\'eron property of  $\mathcal{X}$, this follows from Lemma \ref{empty}.
\end{proof}

\begin{lem}\label{empty}Let $X$ be an $F \llp t \rrp$-definable subset of $\VF^n$.
Assume that $X (F' \llp t \rrp) = \varnothing$
for every field extension $F'$ of $F$. Then
$\EU_{\Gamma, 1} ([X]) = 0$.
\end{lem}

\begin{proof}
Using the notation of (\ref{bbL}) and the isomorphism
(\ref{deco}), we may assume
$X$ is of the form
$[X] = \mathbb{L} (\Psi (a \otimes b))$
with
$a$
in
$K_+ (\RES [m])$
and
$b$ in $K_+ (\Gamma [r])$.
If $r \geq 1$,
$\GEU ([X]) = 0$ by construction of $\GEU$.
Thus, we may assume $r = 0$ and $b = 1$.
Let $k$ be an integer and $Z$ a definable subset in 
$\RES^k$ such that $a = [Z]$. By construction,
$Z$ and $\GEU (X)$ coincide in
$\mathord{!K}(\RES) / ( [\mathbb{A}^1] - 1)$.
On the other hand, if
$X (F' \llp t \rrp) = \varnothing$
for every field extension $F'$ of $F$,
then $Z \cap \kk^k = \varnothing$.
\end{proof}

In particular, we obtain the following:

\begin{cor}[\cite{NS}]\label{cortrace}Let $K = F \llp t \rrp$ with $F$ an algebraically closed field of characteristic zero.
Let $X$ be a smooth proper algebraic variety over $K$.
Then,  for every $m \geq 1$,
\begin{equation}
\tr (\varphi^m; H^{\bullet} (\bar X, \mathbb{Q}_{\ell}))
=
\chi_c (S (X (m))).
\end{equation}
\end{cor}

\begin{proof}By Corollary 7.5.4 of \cite{berket}, for every $q \geq 0$ there are canonical isomorphisms
$H^{q} (\bar X, \mathbb{Q}_{\ell}) \simeq  
H^{q} (\overline{X^{an}}, \mathbb{Q}_{\ell})$.
On the other hand, $X$ being proper, 
$H^{q} (\overline{X^{an}}, \mathbb{Q}_{\ell})$
is canonically isomorphic to
$H^{q}_c (\overline{X^{an}}, \mathbb{Q}_{\ell})$. Let $m \geq 1$. Using Proposition \ref{gentrace} one deduces that 
$\tr (\varphi^m; H^{\bullet} (\bar X, \mathbb{Q}_{\ell})) = \chi_c (\Theta_0 (\EU_{\Gamma, m} ([X])))$ and the result follows from   Proposition \ref{compserre}.
\end{proof}

The original proof  in
Corollary 5.5 \cite{NS}
of
Corollary \ref{cortrace} uses
resolution of singularities, which is not the case of the proof given here.

\begin{rema} Our results also provide
a new construction, not using resolution of singularities, of the motivic Serre invariant of arbitrary algebraic
varieties in equal characteristic zero. This motivic Serre invariant was
constructed in equal characteristic zero and mixed characteristic in Theorem 5.4 of \cite{Nicaise_crelle}, using resolution of singularities,
weak factorization and a refinement of the N\'eron smoothening process to
pairs of varieties. In equal characteristic zero, the trace formula extends to
arbitrary varieties by a formal additivity argument, see Theorem 6.4 and
Corollary 6.5 of  \cite{Nicaise_crelle}.
\end{rema}



\subsection{Analytic variants}
Assume again $R$ is a complete discrete valuation ring, with perfect residue field $F$ and field of fractions $K$.
In \cite{Nicaise_ma}, the construction of the motivic Serre invariant was extended to the class of generic fibers of generically smooth
special formal $R$-schemes.
Special formal $R$-schemes are obtained by gluing formal spectra of quotient of $R$-algebras of the form
$R \{T_1, \ldots, T_r\}\llb S_1, \ldots, S_s\rrb$, cf. \cite{Nicaise_ma}.
In particular, if $\mathcal{X}_{\eta}$ is such a generic fiber and $K = F \llp t \rrp$ with $F$ an algebraically closed field of characteristic zero,
then it follows from Theorem 6.4 of \cite{Nicaise_ma}, generalizing Theorem 5.4 of \cite{NS},
that, with the obvious notations,
for every $m \geq 1$,
\begin{equation}\label{last}
\tr (\varphi^m; H_c^{\bullet} (\bar{\mathcal{X}_{\eta}}, \mathbb{Q}_{\ell}))
=
\chi_c (S (\mathcal{X}_{\eta} (m))).
\end{equation}

In this setting it is natural to replace the theory $\ACVF (0,0)$ considered in the present paper by its rigid analytic variant
$\ACVF^R (0,0)$ introduced by Lipshitz in \cite{Lipshitz} and one may expect that the results from this section still hold for $\ACVF^R (0,0)$-definable sets.
It is quite likely that it should be possible to prove such extensions using arguments similar to ours once 
some
appropriate extension of Theorem \ref{bb}   to this analytic setting is established.
 In particular, one should be able to extend this way Proposition \ref{compserre} and Corollary \ref{cortrace} to 
generic fibers of generically smooth
special formal $R$-schemes. This would provide a proof of (\ref{last}) which would not use resolution of singularities, unlike the original proof in \cite{Nicaise_ma}.

\section{Recovering the motivic zeta function and the motivic Milnor fiber}\label{recov}

\subsection{Some notations and constructions from \cite{HK2}}

Let $A$ be an ordered abelian group and $n$ a non-negative integer.
An $A$-definable subset of $\Gamma^n$ will be called bounded if it is contained in $[-\gamma, \gamma]^n$ for some $A$-definable $\gamma \in \Gamma$.
An $A$-definable subset of $\Gamma^n$ will be called bounded below   if it is contained in $[\gamma, \infty)^n$ for some $A$-definable $\gamma \in \Gamma$.
We recall from \cite{HK2}, Definition 2.4, the definition of various
categories
$\Gamma_A [n]$,
$\Gamma^{\mathrm{bdd}}_A [n]$,
$\mathrm{vol}\Gamma_A [n]$
and
$\mathrm{vol}\Gamma^{\mathrm{bdd}}_A [n]$.
Thus,
$\Gamma_A [n]$ is the category already defined in \S\ref{azer},
$\Gamma^{\mathrm{bdd}}_A [n]$ is the  subcategory of bounded subsets
while
$\mathrm{vol}\Gamma_A [n]$ has the same objects  as $\Gamma_A [n]$
with morphisms $f \colon X \to Y$,
those morphisms in $\Gamma_A [n]$ such that
$\sum_i x_i = \sum_i y_i$ whenever $(y_1, \cdots, y_n) = f (x_1, \cdots, x_n)$,
$\mathrm{vol}\Gamma^{\mathrm{bdd}}_A [n]$
is the subcategory of $\mathrm{vol}\Gamma_A [n]$
whose objects are bounded below.
Finally, we denote by
$\mathrm{vol}\Gamma^{2\mathrm{bdd}}_A [n]$
 the subcategory of $\mathrm{vol}\Gamma_A [n]$ whose objects are  bounded.

We shall also consider the corresponding Grothendieck monoids $K_+ (\Gamma_A [n])$, $K_+ (\Gamma^{\mathrm{bdd}}_A [n])$, $K_+ (\mathrm{vol}\Gamma_A [n])$, $K_+ (\mathrm{vol}\Gamma^{\mathrm{bdd}}_A [n])$, and
$K_+ (\mathrm{vol}\Gamma^{2\mathrm{bdd}}_A [n])$.
We also set $K_+ (\Gamma^{\mathrm{bdd}}_A [\ast]) = \oplus_n K_+ (\Gamma^{\mathrm{bdd}}_A [n])$
with the associated ring $K (\Gamma^{\mathrm{bdd}}_A) $, and similar notation for the other categories.

Let $[0]_1$ denote the class of $\{0\}$ in $K_+ (\Gamma^{\mathrm{bdd}}_A [1])$.
We set \begin{equation}
K_+^{df} (\Gamma^{\mathrm{bdd}}_A) = (K_+ (\Gamma^{\mathrm{bdd}}_A[\ast]) [[0]_1^{-1}])_0,
\end{equation}
where $(K_+ (\Gamma^{\mathrm{bdd}}_A[\ast]) [[0]_1^{-1}])_0$
is the sub-semi-ring  of the graded semi-ring
$K_+ (\Gamma^{\mathrm{bdd}}_A[\ast]) [[0]_1^{-1}]$ consisting of elements of degree $0$.
One defines similarly
$K_+^{df} (\mathrm{vol}\Gamma^{\mathrm{bdd}}_A)$, $K_+^{df} (\mathrm{vol}\Gamma^{2\mathrm{bdd}}_A)$
and denote by
$K^{df} (\Gamma^{\mathrm{bdd}}_A)$, $K^{df} (\mathrm{vol}\Gamma^{\mathrm{bdd}}_A)$ and
$K^{df} (\mathrm{vol}\Gamma^{2\mathrm{bdd}}_A)$
the corresponding rings.

For
$x = (x_1, \dots, x_n) \in \RV^n$, set 
$\wt(x) =  \sum_{1 \leq i \leq n} \val_{\rv} (x_i)$. 
We recall from \cite{HK2}, Definition 3.14, the definition
of the categories $\mathrm{vol}\RV [n]$, $\mathrm{vol}\RES [n]$ and
$\mathrm{vol}\RV^{\mathrm{bdd}} [n]$, given a base structure $A$.
The category 
$\mathrm{vol}\RV [n]$ has the same objects of the category
 $\RV [n]$, namely pairs
$(X, f)$ with $X \subset \RV^*$ and $f \colon X \to \RV^n$ a morphism with finite fibers,
and a morphism $h \colon (X, f) \to (X', f')$ in $\mathrm{vol}\RV [n]$
is a definable bijection $h \colon X \to X'$ such that  $\wt (f (x)) = \wt ((f' \circ h) (x))$
for every $x \in X$.
The category $\mathrm{vol}\RES [n]$ is the full subcategory
of $\mathrm{vol}\RV [n]$ consisting of objects in $\RES [n]$
and  $\mathrm{vol}\RV^{\mathrm{bdd}} [n]$
is
the full subcategory of $\mathrm{vol}\RV [n]$ consisting of objects whose $\Gamma$-image is bounded below.
One defines $\mathrm{vol}\RV^{2\mathrm{bdd}} [n]$ as the subcategory of
$\mathrm{vol}\RV^{\mathrm{bdd}} [n]$ whose
$\Gamma$-image is bounded.
Similar notation as above for the various semi-rings and rings.

We have a map
\begin{equation}
K_+ (\mathrm{vol}\RES[n])
\longrightarrow
K_+ (\mathrm{vol}\RV^{\mathrm{bdd}}[n])
\end{equation}
induced by inclusion and
a map
\begin{equation}
K_+ (\mathrm{vol}\Gamma^{\mathrm{bdd}} [n])
\longrightarrow
K_+ (\mathrm{vol}\RV^{\mathrm{bdd}}[n])
\end{equation}
induced by $X \mapsto \rv^{-1} (X)$.
By \S 3.4 in \cite{HK2}, taking the tensor product,
one gets a canonical morphism
\begin{equation}\label{22}
\Psi \colon K_+ (\mathrm{vol}\RES[\ast]) \otimes
K_+ (\mathrm{vol}\Gamma^{\mathrm{bdd}}[\ast])
\longrightarrow
K_+ (\mathrm{vol}\RV^{\mathrm{bdd}}[\ast])
\end{equation}
whose kernel is the congruence relation generated by pairs
\begin{equation}\label{23}
( [\val_{\rv}^{-1} (\gamma)]_1 \otimes  1, 
1 \otimes [\gamma]_1),
\end{equation}
with $\gamma$ in $\Gamma$ definable.
Here the subscript 1 refers to the fact that the classes are considered in degree $1$.
Note that (\ref{22})
restricts to a morphism
\begin{equation}\label{22b}
\Psi \colon K_+ (\mathrm{vol}\RES[\ast]) \otimes
K_+ (\mathrm{vol}\Gamma^{2\mathrm{bdd}}[\ast])
\longrightarrow
K_+ (\mathrm{vol}\RV^{2\mathrm{bdd}}[\ast]).
\end{equation}

Similarly, cf. Proposition 10.10 of \cite{HK}, there is a
canonical morphism
\begin{equation}\label{25}
\Psi \colon K_+ (\mathrm{vol}\RES[\ast]) \otimes
K_+ (\mathrm{vol}\Gamma[\ast])
\longrightarrow
K_+ (\mathrm{vol}\RV [\ast])
\end{equation}
whose kernel is generated
by the elements (\ref{23}).

Consider the category
$\mathrm{vol}\VF [n]$ of Definition 3.20  in \cite{HK2}
and its bounded version
$\mathrm{vol}\VF^{\mathrm{bdd}} [n]$.
There is a lift of the mapping $\mathbb{L}$ to
a mapping \begin{equation}\mathbb{L}\colon
\mathrm{Ob} \, \mathrm{vol}\RV [n] \longrightarrow \mathrm{Ob} \,  \mathrm{vol}\VF [n] .
\end{equation}
We will denote by 
$I'_{\sp}$   the congruence generated by $[1]_1 = [\RV^{> 0}]_1$
in either $K_+ (\mathrm{vol}\RV [\ast])$ 
or $K_+ (\mathrm{vol} \RV^{\mathrm{bdd}}[\ast])$, or   in one of the monoids $K_+ (\mathrm{vol} \RV [n])$ or $K_+ (\mathrm{vol} \RV^{\mathrm{bdd}} [n])$;
the context will determine the ambient monoid or semi-ring.
By Lemma 3.21 of  \cite{HK2} and Theorems 8.28 and  8.29
of \cite{HK},
there are canonical isomorphisms
\begin{equation}
\int\colon
K_+ (\mathrm{vol} \VF [n])
\longrightarrow
K_+ (\mathrm{vol} \RV [n]) / I'_{\sp}
\end{equation}
and
\begin{equation} \label{8.1.10}
\int\colon
K_+ (\mathrm{vol} \VF^{\mathrm{bdd}} [n])
\longrightarrow
K_+ (\mathrm{vol} \RV^{\mathrm{bdd}} [n]) / I'_{\sp}
\end{equation}
which are characterized by the prescription that,
for $X$ in
$\mathrm{vol} \VF [n]$ and $V$ in
$\mathrm{vol} \RV [n]$ (resp. $\mathrm{vol} \VF^{\mathrm{bdd}} [n]$ and
$\mathrm{vol} \RV^{\mathrm{bdd}} [n]$),
$\int ([X])$ is equal to the class of $[V]$
in
$K_+ (\mathrm{vol} \RV [n]) / I'_{\sp}$ (resp.
$K_+ (\mathrm{vol} \RV^{\mathrm{bdd}} [n]) / I'_{\sp}$)
 if and only if
$[X] = [\mathbb{L} (V)]$. 
We denote similarly the corresponding isomorphisms between Grothendieck rings.

\subsection{The morphisms $h_m$ and $\tilde h_m$}
We go back to the framework of \ref{4.1}, thus the base structure is the field $L_0 = F\llp t \rrp$, with $F$ a trivially valued algebraically closed field of characteristic zero and $\val(t)$ positive and denoted by $1$.
For $\gamma \in \Gamma^n$, let ${\wt(\gamma)} = \sum_{1 \leq i \leq n} \gamma_i$.

Let $\mathbb{Z} [T, T^{-1}]_{\mathrm{loc}}$ denote the localisation of the
ring of Laurent polynomials $\mathbb{Z} [T, T^{-1}]$ with respect to the multiplicative family generated by the polynomials $1 - T^i$, $i \geq 1$.

Let $\Delta$ be a bounded definable subset of $\Gamma^n$.  
For every integer $m \geq 1$,
we set
\begin{equation}\label{alphasum}
\alpha_m (\Delta) = \tcb{(T - 1)^n}
\sum_{(\gamma_1, \dots, \gamma_n) \in \Delta \cap (m^{-1} \mathbb{Z})^n}
T^{- m {\wt(\gamma)}} 
\end{equation}
in
$\mathbb{Z} [T, T^{-1}]$.

Assume now $\Delta$ is a bounded below definable subset of $\Gamma^n$. 
The sum (\ref{alphasum}) is no longer finite, but it still makes sense has a Laurent  series, since in (\ref{alphasum})  only a finite number of terms have a given weight since $\Delta$ is bounded below.
\begin{lem}\label{sbrat}Let $\Delta$ be a bounded below definable subset of $\Gamma^n$.  
For every integer $m \geq 1$,
the Laurent series
\begin{equation}\label{alphasum2}
\tilde \alpha_m (\Delta) =  \tcb{(T - 1)^n}
\sum_{(\gamma_1, \dots, \gamma_n) \in \Delta \cap (m^{-1} \mathbb{Z})^n}
T^{- m {\wt(\gamma)}} 
\end{equation}
belongs to 
$\mathbb{Z} [T, T^{-1}]_{\mathrm{loc}}$.
\end{lem}

\begin{proof}It is enough to prove the result for $m = 1$.
We may assume $\Delta$ is convex and closed. Thus, it is 
the convex hull of a finite family of rational half-lines and points in $\mathbb{Q}^n$, i.e. a rational polytope according to  the terminology of \cite{Brion}.
Consider the formal series
\[
\Phi_{\Delta} (T_1, \cdots, T_n) :=
\sum_{(\gamma_1, \dots, \gamma_n) \in \Delta \cap \mathbb{Z}^n} \prod T_i^{\gamma_i}.
\]
It follows from \cite{Brion} and \cite{Ishida} that
$\Phi_{\Delta} (T_1, \cdots, T_n)$ belongs to the localisation of 
$\mathbb{Z} [T_1, T_1^{-1}, \cdots, T_n, T_n^{-1}]$ with respect to the multiplicative family generated by
$1 - \prod T_i^{\gamma_i}$, $(\gamma_1, \dots, \gamma_n) \in \mathbb{Z}^n \setminus \{0\}$.
Indeed, the core of the paper \cite{Brion} deals with integral polytopes, but in its
\S 3.3  it is explained how  to deduce the statement for rational polytopes. Since $\Delta$ is bounded below, 
$\Phi_{\Delta} (T_1, \cdots, T_n)$ belongs in fact to the localisation of 
$\mathbb{Z} [T_1, T_1^{-1}, \cdots, T_n, T_n^{-1}]$ with respect to the multiplicative family generated by
$1 - \prod T_i^{\gamma_i}$, $(\gamma_1, \dots, \gamma_n) \in \mathbb{N}^n \setminus \{0\}$.
Thus one may consider the restriction of $\Phi_{\Delta}$ to the line $T = T_1 = \cdots = T_n$
which belongs to $\mathbb{Z} [T, T^{-1}]_{\mathrm{loc}}$ and is equal to 
$\tilde \alpha_m (\Delta)$ up to the factor $\tcb{(T - 1)^n}$.
\end{proof}

Let $\mathord{!K} ( \RES) ([{\mathbb{A}^1}]^{-1})_{\mathrm{loc}}$ denote the localisation of $\mathord{!K} ( \RES) ([{\mathbb{A}^1}]^{-1})$ with respect to the multiplicative family generated by the   
elements $1 - [{\mathbb{A}^1}]^i$, $i \geq 1$.
There are unique morphisms
$\theta : \mathbb{Z} [T, T^{-1}] \to \mathord{!K} ( \RES) ([{\mathbb{A}^1}]^{-1})$
and 
$\tilde \theta : \mathbb{Z} [T, T^{-1}]_{\mathrm{loc}} \to \mathord{!K} ( \RES) ([{\mathbb{A}^1}]^{-1})_{\mathrm{loc}}$
sending $T$ to $[{\mathbb{A}^1}]$.

If $\Delta$ is a bounded, resp. bounded below, definable subset of $\Gamma^n$,
we set
$a_m (\Delta) = \theta (\alpha_m (\Delta))$, resp.
$\tilde a_m (\Delta) = \tilde \theta (\tilde \alpha_m (\Delta))$.
By additivity, this gives rise to morphisms
\begin{equation}
a_m \colon K (\mathrm{vol} \Gamma^{2\mathrm{bdd}} [\ast])
\longrightarrow
\mathord{!K} ( \RES) ([{\mathbb{A}^1}]^{-1})
\end{equation}
and
\begin{equation}
\tilde a_m \colon K (\mathrm{vol} \Gamma^{\mathrm{bdd}} [\ast])
\longrightarrow
\mathord{!K} ( \RES) ([{\mathbb{A}^1}]^{-1})_{\mathrm{loc}}.
\end{equation}

Now consider $X = (X, f)$ in $\RES [n]$. Let $\gamma = (\gamma_1, \dots, \gamma_n)$ and 
assume $f (X) \subset V_{\gamma_1} \times \dots \times V_{\gamma_n}$.
Set
\begin{equation}
b^0_m (X) = [X]  
\Bigl(\frac{[1]_1}{[{\mathbb{A}^1}]}\Bigr)^{ m {\wt(\gamma)}}
\end{equation}
in
$
\mathord{!K} (\RES [\ast]) ([{\mathbb{A}^1}]^{-1})
$ if $m (\gamma_1, \dots, \gamma_n) \in \mathbb{Z}^n$
and
$b^0_m (X) = 0$ otherwise.
 Note that $f (X) =  f(X) \cap \RES_m$ in the first case.
This construction extends uniquely to a morphism
\begin{equation}
b^0_m \colon K (\mathrm{vol} \RES [\ast])
\longrightarrow
\mathord{!K}(\mathrm{vol} \RES [\ast]) ([{\mathbb{A}^1}]^{-1}).
\end{equation}
By composing $b^0_m $ with the canonical forgetful morphism
$\mathord{!K}(\mathrm{vol} \RES [\ast]) ([{\mathbb{A}^1}]^{-1}) \to \mathord{!K}( \RES) ([{\mathbb{A}^1}]^{-1})$,
one gets a morphism
 \begin{equation}
b_m \colon K (\mathrm{vol} \RES [\ast])
\longrightarrow
\mathord{!K}( \RES) ([{\mathbb{A}^1}]^{-1}).
\end{equation}
One denotes by   $\tilde b_m$ the morphism 
\begin{equation}
\tilde b_m \colon K (\mathrm{vol} \RES [\ast])
\longrightarrow
\mathord{!K}( \RES) ([{\mathbb{A}^1}]^{-1})_{\mathrm{loc}}
\end{equation}
obtained by composing
$b_m$ with the localisation morphism
$\mathord{!K}( \RES) ([{\mathbb{A}^1}]^{-1}) \to
\mathord{!K}( \RES) ([{\mathbb{A}^1}]^{-1})_{\mathrm{loc}}$.

 The  morphism
\begin{equation}b_m \otimes a_m \colon
K (\mathrm{vol} \RES [\ast]) \otimes
K  (\mathrm{vol}\Gamma^{2\mathrm{bdd}}[\ast])
\longrightarrow
\mathord{!K} (\RES) ([{\mathbb{A}^1}]^{-1})
\end{equation}
factors through the relations (\ref{23})
and gives rise to a morphism
\begin{equation}
h_m \colon K(\mathrm{vol} \RV^{2\mathrm{bdd}} [\ast])
\longrightarrow
\mathord{!K} (\RES) ([{\mathbb{A}^1}]^{-1}).
\end{equation}
Indeed,
if $\gamma = i / m$, then
$a_m ([\gamma]_1) = (\frac{1}{[{\mathbb{A}^1}]})^{i} \tcb{([\mathbb{A}^1] - 1)}$
and $[\val^{-1}_{\rv} (\gamma)]_1 = [{\mathbb{A}^1}]- [1]_1$
in
$\mathord{!K}_+(\mathrm{vol} \RES [1])$, thus
$a_m ([\gamma]_1) = b_m ([\val^{-1}_{\rv} (\gamma)]_1)$.
Similarly, 
the  morphism
\begin{equation}\tilde b_m \otimes \tilde a_m \colon
K (\mathrm{vol} \RES [\ast]) \otimes
K  (\mathrm{vol}\Gamma^{\mathrm{bdd}}[\ast])
\longrightarrow
\mathord{!K} (\RES) ([{\mathbb{A}^1}]^{-1})_{\mathrm{loc}}
\end{equation}
gives rise to a morphism
\begin{equation}
\tilde h_m \colon K(\mathrm{vol} \RV^{\mathrm{bdd}} [\ast])
\longrightarrow
\mathord{!K} (\RES) ([{\mathbb{A}^1}]^{-1})_{\mathrm{loc}}
\end{equation}
and the diagram
\begin{equation}\xymatrix{
K(\mathrm{vol} \RV^{2\mathrm{bdd}} [\ast]) \ar[r]^{h_m} \ar[d]_{}&\mathord{!K} (\RES) ([{\mathbb{A}^1}]^{-1})\ar[d]^{}\\
K(\mathrm{vol} \RV^{\mathrm{bdd}} [\ast]) \ar[r]^{\tilde h_m} & \mathord{!K} (\RES) ([{\mathbb{A}^1}]^{-1})_{\mathrm{loc}}}
\end{equation}
is commutative.

\begin{lem}For every $m \geq 1$, the morphism $\tilde h_m$ vanishes on the congruence $I'_{\sp}$.
\end{lem}

\begin{proof}
Indeed, if $\ell$ denotes the open half-line $(0, \infty)$ in  $\Gamma$,
$\tilde \alpha_m (\ell) = \tcb{(T - 1)} \sum_{i >0} T^{-i} = \tcb{1}$,
therefore $\tilde h_m ([\RV^{>0}]_1) = \tcb{1}$.
On the other hand, $h_m ([1]_1) = \tcb{1}$ by definition.
\end{proof}

It follows that the morphism
$\tilde h_m$ factors
 through a morphism
\begin{equation}
\tilde h_m \colon K (\mathrm{vol} \RV^{\mathrm{bdd}} [\ast]) / I'_{\sp}
\longrightarrow
\mathord{!K} ( \RES) ([{\mathbb{A}^1}]^{-1})_{\mathrm{loc}}.
\end{equation}
In particular, if
$\alpha$ and $\alpha'$ are two elements in
$ K (\mathrm{vol} \RV^{2\mathrm{bdd}} [\ast])$ with same image in
$ K  (\mathrm{vol} \RV [\ast]) /I'_{\sp}$, then $h_m (\alpha)$ and $h_m (\alpha')$
have the same image in $\mathord{!K} ( \RES) ([{\mathbb{A}^1}]^{-1})_{\mathrm{loc}}$.

Let us now state the analogue of Proposition \ref{3} in this context.

\begin{prop}\label{3vol}Let $m$ be a positive integer. Let  $n$ and $r$ be integers, let $\beta \in \G^n$ and
let $X$ be a $\b$-invariant $F \llp t \rrp$-definable subset of $\Oo^n \times \RV^{r}$. We assume that
$X$ is contained in $\VF^n \times W$ with $W$ a boundedly imaginary definable subset of $\RV^{r}$, and that $X_w$ is bounded, for every $w \in W$. We also assume that the projection $X \to \VF^n$ has finite fibers.
Then $\tilde h_m (\int ([X]))$ is equal to the image of the class $\widetilde X [m]$ as defined in \S \ref{4.2} in
$\mathord{!K} ( \RES) ([{\mathbb{A}^1}]^{-1})_{\mathrm{loc}}$.
\end{prop}

\begin{proof}  Since both sides are invariant under the transformations of Proposition \ref{2},
we may assume by Proposition \ref{2} that
there exists a
definable boundedly imaginary subset $H$ of $\RV^{r'}$
 and a map $h \colon \{1, \ldots, n \} \to\{1, \ldots, r' \}$
 such that
 \begin{equation}
 X = \{(a, b); b \in H, \rv (a_i) = b_{h (i)}, 1 \leq i \leq n \}
 \end{equation}
and
 the map  $r \colon H \to \RV^{n}$ given by
 $b \mapsto (b_{h(1)}, \ldots, b_{h (n)})$ is finite to one.
According to
(\ref{22b}) we may assume $[H] = \Psi ([W] \otimes [\Delta])$
with  $W$ in $\RES [r]$ and
$\Delta$ bounded  in $\Gamma [n - r]$.
By induction on dimension and considering products, it is enough to prove the result when $X$ is the lifting of $W$ or $\Delta$.
In both cases, this is clear by construction.
 \end{proof}

\begin{rema}\label{rema:vol}
 The definition of the morphisms of $\mathrm{vol}\VF [n]$ refers implicitly to the standard volume form
on $K^n$, restricted to $\Oo^n$.  When an $n$-dimensional variety is given without a specific embedding, we must specify a volume form
 since,  in principle, integrals depend on the form,
 up to multiplication by a definable function into $\mathbb{G}_m(\Oo)$.   
However, when $V$ is a smooth variety over $F$, with a volume form $\omega$ (a nowhere vanishing section of $\bigwedge^{\rm{top}} T V$) defined over $F$, and
$X$ is a bounded, $\b$-invariant $F\llp t \rrp$-definable subset of $V(\Oo)$,
  then $\int ([X])$ does not depend on the choice of $\omega$, as long as $\omega$ is chosen over $F$.  The reason is 
 that given another such form $\omega' $, we have   $\omega' = g \omega$ for some non-vanishing
  regular functions $g$ on $V$, defined over $F$. Thus, denoting by $\mathrm{red}$ the reduction mapping $V (\Oo) \to V$,  for  $u \in V$ we have $\mathrm{red} (g (u)) = g(\mathrm{red} (u)) \neq 0$ so $\val (g(u))=0$.
  In particular, we  shall refer to  $\int ([X]) \in K ({\rm{vol}} \RV [n])$ in this setting without further mention of the volume form.
\end{rema}


\subsection{Expressing the motivic zeta function}
Let $K^{\hat \mu}(\Var_F)_{\mathrm{loc}}$ denote the localisation of $K^{\hat \mu}(\Var_F)$ with respect to the multiplicative family generated by $[\mathbb{A}^1]$ and the   
elements $1 - [{\mathbb{A}^1}]^i$, $i \geq 1$. One defines similarly $K (\Var_F)_{\mathrm{loc}}$.
The isomorphism $\Theta$ of  (\ref{4.3.2}) induces  isomorphisms
\begin{equation}
 \Theta \colon \, \mathord{!K}(\RES)[ [\mathbb{A}^1]^{-1}]  \longrightarrow \, K^{\hat \mu}(\Var_F)[ [\mathbb{A}^1]^{-1}] \end{equation}
and
\begin{equation}
\Theta \colon \, \mathord{!K}(\RES)[ [\mathbb{A}^1]^{-1}]_{\mathrm{loc}}  \longrightarrow \, K^{\hat \mu}(\Var_F)_{\mathrm{loc}}.
\end{equation}

Let $X$ be a smooth connected algebraic variety of dimension $d$ over $F$
and $f$ a non-constant regular function
$f \colon X \to \mathbb{A}^1_F$.

For any $m \geq 1$, we consider 
  $ \mathcal{X}_{m, x}$ as defined in  (\ref{1.1.4})
\begin{equation}
\mathcal{X}_{m, x} =
\{\varphi \in X (\mathbb{C} [t] / t^{m + 1}); f (\varphi) = t^m \mod t^{m + 1}, \varphi (0) = x\}
\end{equation}
and $\mathcal{X}_x$ from Corollary
\ref{4}
\begin{equation}
\mathcal{X}_x  = \{y \in X (\Oo); \rv f(y) = \rv(t) \, \mathrm{and} \, \pi (y) = x \}.
\end{equation}

Recall that $\mathcal{X}_x$ is $\beta$-invariant for $\beta >0$.
After replacing $X$ by an affine open containing $x$, we may assume the existence of a  volume form on $X$ defined over $F$.
Thus, using the convention in 
Remark \ref{rema:vol}, we may consider
$\tilde h_m (\int ([\mathcal{X}_x]))$
in 
$ {\mathord{!K}(\RES})[ [\mathbb{A}^1]^{-1}]_{\mathrm{loc}}$.

We have the following interpretation for the class of
$ \mathcal{X}_{m, x}$.

\begin{prop}\label{4vol}
Let $X$ be a smooth  variety over $F$, $f$ be a regular function on $X$ and $x$ be a closed point of
$f^{-1} (0)$.
Then, for every integer $m \geq 1$,
\[
\Theta \Bigl(\tilde h_m  \Bigl( \int ([\mathcal{X}_x]) \Bigr)   \Bigr)
=
 [\mathcal{X}_{m, x}] \, [\mathbb{A}^{\tcb{m d}}]^{-1}
\]
in
$K^{\hat \mu}(\Var_F)_{\mathrm{loc}}$.
\end{prop}

\begin{proof}By definition, using notation from \ref{4.2},
\[ \widetilde{\mathcal{X}_x} [m] = [\mathcal{X}_x [m; 1 + 1/m]] \, [\mathbb{A}^{\tcb{m d}}]^{-1}.\]
It follows from Proposition \ref{3vol}, by a similar argument as the one in the proof of
Corollary
\ref{4}, that 
$\tilde h_m  ( \int ([\mathcal{X}_x ]))$ and $ \widetilde{\mathcal{X}_x} [m]$
have the same image in
$\mathord{!K}( \RES ) ([{\mathbb{A}^1}]^{-1})_{\mathrm{loc}}$.
On the other hand,
since, as already observed in  \ref{pft},
$\mathcal{X}_{m, x}$ is isomorphic to $ \{\varphi \in X (\mathbb{C} [t^{1/m}] / t^{(m+1)/m}) ; \rv( f(\varphi)) = \rv(t),  \, \varphi (0) = x\}$,
$\mathcal{X}_{m, x}$  and
$\Theta ([\mathcal{X}_x [m; 1 + 1/m]])$ have the same class in 
in $K^{\hat \mu}(\Var_F)$.
The result follows.
\end{proof}

The motivic zeta function $Z_{f, x}  (T)$ attached to $(f, x)$ is
the following generating function, cf. \cite{motivic}, \cite{barc},
\begin{equation}
Z_{f, x}  (T) = \sum_{m \geq 1} [\mathcal{X}_{m, x}] \, [\mathbb{A}^{md}]^{-1} \, T^m
\end{equation}
in
$K^{\hat \mu}(\Var_F) [ [\mathbb{A}^1]^{-1}] \llb T \rrb$.

Let $\iota \colon K^{\hat \mu}(\Var_F) [ [\mathbb{A}^1]^{-1}] \to K^{\hat \mu}(\Var_F)_{\mathrm{loc}}$
denote the localisation morphism.
Applying $\iota$ termwise to $Z_{f, x}  (T)$
we obtain a series $\tilde Z_{f, x}$ in
$K^{\hat \mu}(\Var_F)_{\mathrm{loc}} \llb T \rrb$.

Thus, by Proposition \ref{4vol}, $\tilde Z_{f, x}  (T)$ may  be expressed directly in terms of
$\mathcal{X}_x$:

\begin{cor}\label{5vol}
Let $X$ be a smooth  variety over $F$of dimension $d$, $f$ a regular function on $X$ and $x$ a closed point of
$f^{-1} (0)$.
Then,
\[
\tilde Z_{f, x}  (T)
=  
 \sum_{m \geq 1}
\Theta \Bigl(\tilde h_m  \Bigl( \int ([\mathcal{X}_x]) \Bigr)   \Bigr) \, T^m. \]
\end{cor}


\subsection{Rational series}
Let $R$ be a ring and let $\Aa$ be an invertible element in $R$.
We consider the ring $R [T]_{\dagger}$ (resp. 
$R [T, T^{-1}]_{\dagger}$)
which is the 
localization of $R [T]$ (resp. 
$R [T, T^{-1}]$) with respect to the multiplicative
family generated by $1 - \Aa^a T^b$, $a \in \Zz$, $b \geq 1$.
By expanding into powers in $T$ one gets a morphism
\begin{equation}
{e}_T  \colon R [T]_{\dagger} \longrightarrow R \llb T \rrb
\end{equation}
which is easily checked to be injective. 
We shall  identify an element in $R [T]_{\dagger}$ 
with its image in $R \llb T \rrb$.
If $h = P/Q$  belongs to 
$R [T]_{\dagger}$, the difference
$ \deg (P) - \deg (Q)$ depends only on $h$, thus will be denoted $\deg (h)$.
If $\deg (h) \leq 0$,
we define
$\lim_{T \to \infty} h$ as follows.
If $\deg (h) < 0$, we set $\lim_{T \to \infty} h = 0$.
If $h = P / Q$ with $P$ and $Q$ of degree $n$, let $p$ and $q$ be the leading coefficients of $P$ and $Q$.
Since $q$ is of the form $\varepsilon \Aa^a$ for some $a \in \Zz$ and $\varepsilon \in \{-1, 1\}$, we may set
$\lim_{T \to \infty} h = p \varepsilon \Aa^{-a}$, which is independent from the choice of $P$ and $Q$.

Since
\begin{equation}
\frac{1}{1 - \Aa^a T^b} = - \frac{ \Aa^{-a} T^{-b}}{1 -\Aa^{-a} T^{-b}},
\end{equation}
one may also expand elements of $R [T]_{\dagger}$
into powers of $T^{-1}$, giving rise to a morphism
\begin{equation}
{e}_{T^{-1}} \colon R [T]_{\dagger} \longrightarrow R \llb T^{-1} \rrb [T].
\end{equation}
In particular,  if $h$  belongs to 
$R [T]_{\dagger}$, $\deg (h) \leq 0$ if and only if ${e}_{T^{-1}} (h)$ belongs to 
$R \llb T^{-1} \rrb $. Furthermore, in this case
$\lim_{T \to \infty} h$ is equal to 
the constant term of ${e}_{T^{-1}} (h)$.

If $f (T) = \sum_{n \geq 0} a_n T^n$ and
$g (T) = \sum_{n \geq 0} b_n T^n$ are two series
in 
$R \llb T \rrb$ one defines their Hadamard product as
$(f \ast g)  (T) = \sum_{n \geq 0} a_n b_n T^n$.

\begin{lem}\label{convol}
Let $h$ and $h'$ belong to  $R [T]_{\dagger}$.
Set $\varphi = {e}_T (h)$,
$\varphi' = {e}_T (h')$.
\begin{enumerate}
\item[(1)]There exists a \textup{(}unique\textup{)} element $\tilde h$ in $R [T]_{\dagger}$
such that
${e}_T (\tilde h) = \varphi \ast \varphi'$.
\item[(2)]Assume that $\varphi$ and $\varphi' $ belong to $TR \llb T \rrb$,
and that
$\deg (h), \deg (h') \leq 0$.
Then 
$\deg (\tilde h) \leq 0$
and
\[
\lim_{T \to \infty} \tilde h = - \lim_{T \to \infty} h  \cdot \lim_{T \to \infty} h'.
\]
\end{enumerate}
\end{lem}

\begin{proof}Assertion (1) follows from 
Propositions 5.1.1 and 5.1.2 of \cite{DLduke} and their proofs.
Indeed, by (the proof of) Proposition 5.1.1  of \cite{DLduke},
there exists $\tilde h \in R [T, T^{-1}]_{\dagger}$
such that ${e}_T (\tilde h)  =  \varphi \ast \varphi'$ (with ${e}_T$ extended to a morphism
$R [T, T^{-1}]_{\dagger} \to R\llb T \rrb[T^{-1}]$).
But this forces $\tilde h $ to belong in fact to  $R [T]_{\dagger}$.
By (the proof of) Proposition 5.1.1  of \cite{DLduke}, cf. also  Proposition 5.1.2 of \cite{DLduke}
and its proof, it follows from the assumptions in (2) that 
\[
({e}_{T^{-1}} (\tilde h)) (T^{-1}) = - ({e}_{T^{-1}} (h)) (T^{-1}) \ast ({e}_{T^{-1}}  (h')) (T^{-1}).
\]
Thus
$\deg (\tilde h) \leq 0$
and
$
\lim_{T \to \infty} \tilde h = - \lim_{T \to \infty} h  \cdot \lim_{T \to \infty} h'$.
\end{proof}

When $\varphi \in R \llb T \rrb$ is of the form
${e}_T (h)$ with $h \in R [T]_{\dagger}$, we shall say
$\lim_{T \to \infty} \varphi$ {\em exists} if 
$\deg (h) \leq 0$, and set
$\lim_{T \to \infty} \varphi = \lim_{T \to \infty} h$.

\subsection{Expressing  the motivic Milnor fiber}\label{emmf}

We consider 
the rings
$K^{\hat \mu}(\Var_F) [ [\mathbb{A}^1]^{-1}] [T]_{\dagger}$
and
$K^{\hat \mu}(\Var_F)_{\mathrm{loc}} [T]_{\dagger}$
with $\Aa =  [\mathbb{A}^1]$.
More generally, in this section, when we write $R[T]_{\dagger}$
it will always be with $\Aa =  [\mathbb{A}^1]$.

It is known that the motivic zeta function $Z_{f, x} (T)$ belongs to 
$K^{\hat \mu}(\Var_F) [ [\mathbb{A}^1]^{-1}][T]_{\dagger}$ and that
$\lim_{T \to \infty} Z_{f, x} (T)$ exists, cf. \cite{barc}, \cite{seattle}.

One sets
\begin{equation}\label{def:mmf}
\mathcal{S}_{f, x}= - \lim_{T \to \infty} Z_{f, x}  (T).
\end{equation}
This element of $K^{\hat \mu}(\Var_F) [ [\mathbb{A}^1]^{-1}]$  is the motivic Milnor fiber considered in \cite{barc}, \cite{seattle}.
We shall  show in Corollary \ref{limitcor} how one may extract directly
the image of $\mathcal{S}_{f, x}$ in $K^{\hat \mu}(\Var_F)_{\mathrm{loc}}$ from $\int ([\mathcal{X}_{x}])$.

\medskip

Let $\chi$ denote the o-minimal Euler
characteristic. 
There exists a unique morphism
\begin{equation}
\alpha \colon
K (\mathrm{vol} \Gamma [\ast]) \longrightarrow
\mathord{!K} ( \RES) ([{\mathbb{A}^1}]^{-1})
\end{equation}
which, for every $n \geq 0$,  sends the class of $\Delta$ in
$K (\mathrm{vol} \Gamma [n])$
to $\chi (\Delta) \tcb{( [{\mathbb{A}^1}] - 1)^n}$, and a unique morphism
\begin{equation}
\beta \colon
K (\mathrm{vol} \RES [\ast]) \longrightarrow
\mathord{!K} ( \RES) ([{\mathbb{A}^1}]^{-1})
\end{equation}
which, for every $n \geq 0$,   sends the class of $Y$ in $K (\mathrm{vol} \RES [n])$
to $[Y]$.  

Taking the tensor product of $\alpha$ and $\beta$
one gets a morphism
\begin{equation}
\Upsilon \colon
K (\mathrm{vol} \RV [\ast]) \longrightarrow
\mathord{!K} ( \RES ) ([{\mathbb{A}^1}]^{-1})
\end{equation}
since the relations (\ref{23}) in the kernel of
 the morphism (\ref{25}) are respected.
One defines similarly a   morphism
\begin{equation}\Upsilon \colon K (\mathrm{vol} \RV^{2\mathrm{bdd}} [\ast]) \longrightarrow
\mathord{!K} ( \RES ) ([{\mathbb{A}^1}]^{-1}).\end{equation}

\begin{prop}\label{limit}Let $Y$ be in $K (\mathrm{vol} \RV^{2\mathrm{bdd}} [\ast])$.
The series
\[Z (Y)   (T)
= \sum_{m \geq 1}
 h_m  (Y)    \, T^m
\]
in $\mathord{!K} (\RES) ([{\mathbb{A}^1}]^{-1}) \llb T \rrb$
belongs to $\mathord{!K} (\RES) ([{\mathbb{A}^1}]^{-1}) [T]_{\dagger}$, $\lim_{T \to \infty} Z (Y)   (T) $ exists  and
\[
\lim_{T \to \infty} Z (Y)   (T)
=
 - \Upsilon (Y).
\]
\end{prop}

\begin{proof}
We may assume $Y$ is of the form
$\Psi ([W ]\otimes [\Delta])$
with $W$ in  $\RES [p]$ and
$\Delta$ in $\Gamma [q]$.
By Lemma \ref{convol},
$Z (Y) (T)$ is the Hadamard product
of
$Z (\Psi ([W ]\otimes 1)) (T)$
and
$Z (\Psi (1\otimes [\Delta] ))(T)$.
Thus it  is enough to prove the statement for
$\Psi ([W ]\otimes 1)$ and
$\Psi (1\otimes [\Delta])$.
By construction,
\begin{equation}Z (\Psi ([W ]\otimes 1)) (T)
=
[W] 
\sum_{m \geq 1} [\Aa^1]^{ - \alpha m} \, T^{\beta m}\end{equation}
for some integers $\alpha \in \Zz$ and $\beta \geq 1$.
Hence 
$Z (\Psi ([W ]\otimes 1)) (T)$ belongs to
$\mathord{!K} ( \RES ) ([{\mathbb{A}^1}]^{-1}) [T]_{\dagger}$,
 $\lim_{T \to \infty} Z (\Psi ([W ]\otimes 1)) (T) $ exists and  is equal to
 $- [W]  
 = 
 - \Upsilon (\Psi ([W ]\otimes 1))$.

The statement for
$\Psi (1\otimes [\Delta])$ follows from Lemma \ref{polytope}, using
the morphism
$\mathbb{Z} [U, U^{-1}] \to
\mathord{!K} ( \RES) ([{\mathbb{A}^1}]^{-1})$ sending $U$ to $[{\mathbb{A}^1}]^{-1}$.
\end{proof}

\begin{lem}\label{polytope}
Let $\Delta$ be a bounded definable subset of $\Gamma^n$.  Let $\ell \colon \Delta \to \Gamma$
be piecewise \textup{(}i.e. on each piece of a finite definable partition\textup{)} of the form
$x = (x_i) \mapsto \sum a_i x_i + b$, with the $a_i$'s and $b$ in $\mathbb{Z}$.
For every integer $m \geq 1$,
set
\[
s_m (\Delta, \ell) =
\sum_{(\gamma_1, \dots, \gamma_n) \in \Delta \cap (m^{-1} \mathbb{Z})^n}
U^{- m {\ell(\gamma)}},
\]
in
$\mathbb{Z} [U, U^{-1}]$
and set
\[
Z (\Delta, \ell) (T) = \sum_{m\geq 1} s_m (\Delta, \ell) \, T^m
\]
in $\mathbb{Z} [U, U^{-1}] \llb T \rrb$.
Then, the series $Z (\Delta, \ell)$
belongs to
$\mathbb{Z} [U, U^{-1}] [T]_{\dagger}$, $\lim_{T \to \infty} Z (\Delta, \ell) (T) $ exists and
\[
\lim_{T \to \infty} Z (\Delta, \ell) (T)
=
 - \chi (\Delta).
\]
\end{lem}

\begin{proof}Let $\Delta$ be a bounded definable subset of $\Gamma^n$. We shall say the
lemma holds for $\Delta$ if it holds for $\Delta$ and any $\ell$.
If $\Delta'$ is  another bounded definable subset of $\Gamma^n$
such that $[\Delta] = [\Delta']$ in
$K_+ (\Gamma^{\mathrm{bdd}}_{\mathbb{Z}}) [n]$, then the lemma holds for
$\Delta$ if and only it holds for $\Delta'$. Thus the property  for $\Delta$
depends only on its  class in $K_+ (\Gamma^{\mathrm{bdd}}_{\mathbb{Z}}) [n]$.
Localisation with respect to
$[0]_1$ is harmless here, and one deduces that  the property of satisfying the lemma for $\Delta$
depends only on its  class
$[\Delta]/ [0]_1^n$
in
$K_+^{df} (\Gamma^{\mathrm{bdd}}_{\mathbb{Z}})$.
We shall say $[\Delta]/ [0]_1^n$
satisfies the lemma if $\Delta$ does.

Let $I$ be a definable bounded interval in
$\Gamma$ and $\ell' \colon \Gamma \to \Gamma$ a  linear form
$x \mapsto a x + b$ with $a$ and $b$ in $\mathbb{Z}$.
Then,  by a direct geometric series computation one gets that the lemma holds for
$I$ and $\ell'$. It follows that the lemma holds for $I$ and any $\ell$.
In particular, the lemma holds for the subsets
$[0, \gamma)$ and $\{\gamma\}$ of $\Gamma$, with $\gamma$  in $\mathbb{Q}$.
Let $K_+^{df} (\Gamma^{\mathrm{bdd}}_{\mathbb{Z}})'$ be the sub-semi-ring of
$K_+^{df} (\Gamma^{\mathrm{bdd}}_{\mathbb{Z}})$ generated by
$[\gamma]_1 / [0]_1$ and
$[0, \gamma)_1 / [0]_1$, for
$\gamma$  in $\mathbb{Q}$.
It follows from Lemma \ref{convol}
that the lemma holds for all elements in
$K_+^{df} (\Gamma^{\mathrm{bdd}}_{\mathbb{Z}})'$ since it holds
for the generators $[\gamma]_1 / [0]_1$ and
$[0, \gamma)_1 / [0]_1$.
By Lemma 2.21 of \cite{HK2}, for any element $a$
in
$K_+^{df} (\Gamma^{\mathrm{bdd}}_{\mathbb{Z}})$
there exists a nonzero
$m \in \mathbb{N}$, $b$ and $c$ in
$K_+^{df} (\Gamma^{\mathrm{bdd}}_{\mathbb{Z}})'$
such that $m a + b = c$. Since the lemma holds for
$b$ and $c$, it follows that the lemma holds for $ma$, hence for $a$, and the statement follows.
\end{proof}

\begin{cor}\label{limitcor}
Let $X$ be a smooth  variety over $F$, $f$ a regular function on $X$ and $x$ a closed point of
$f^{-1} (0)$.
Then the image of $\mathcal{S}_{f, x}$
in $K^{\hat \mu}(\Var_F)_{\mathrm{loc}}$ is equal to 
\[
\Theta \Bigl(\Upsilon \Bigl(\int ([\mathcal{X}_x])\Bigr) \Bigr).
\]
\end{cor}

\begin{proof}
This follows directly from Corollary \ref{5vol} and Proposition \ref{limit}.
\end{proof}


\begin{rema}\label{hfactors}It is not known whether the localisation morphisms
$\iota \colon K(\Var_F) [ [\mathbb{A}^1]^{-1}] \to K(\Var_F)_{\mathrm{loc}}$
and
$\iota \colon K^{\hat \mu}(\Var_F) [ [\mathbb{A}^1]^{-1}] \to K^{\hat \mu}(\Var_F)_{\mathrm{loc}}$
are injective.
However, the morphism $H \colon K(\Var_F) [ [\mathbb{A}^1]^{-1}]  \to \mathbb{Z} [u, v, u^{-1}, v^{-1}]$
induced by the Hodge-Deligne polynomial vanishes on the kernel of $\iota$, hence factors through the image of $\iota$.
In particular, the Euler characteristic with compact supports $\chi_c \colon K(\Var_F) [ [\mathbb{A}^1]^{-1}]  \to \mathbb{Z} $
factors through the image of $\iota$. This extends to the equivariant setting, in particular one can
recover the Hodge-Steenbrink spectrum of $f$ at $x$ from
the image of
$\mathcal{S}_{f, x}$ in $K^{\hat \mu}(\Var_F)_{\mathrm{loc}}$, cf. \cite{barc}, \cite{seattle}.
\end{rema}

\begin{rema}\label{carfibre}When $F = \mathbb{C}$, Theorem \ref{mt} together with Corollary \ref{limitcor} provides a proof avoiding resolution of singularities
that the topological Milnor fiber $F_x$ and the motivic Milnor fiber $\mathcal{S}_{f, x}$ have the same Euler characteristic with compact supports, namely that
$\chi_c (F_x) = \chi_c (\mathcal{S}_{f, x})$.
Indeed, by Remark \ref{hfactors} one may apply  $\chi_c$ to
(\ref{def:mmf}), thus getting
$\chi_c (\mathcal{S}_{f, x}) = - \lim_{T \to \infty} \sum_{m \geq 1}  \chi_c (\mathcal{X}_{m,x}) T^m$,
which may be rewritten,  by Theorem \ref{mt}, as
$
\chi_c (\mathcal{S}_{f, x}) = - \lim_{T \to \infty} \sum_{m \geq 1}  \Lambda (M_x^m) T^m$.
By quasi-unipotence of local monodromy (a statement for which there exist proofs
not using resolution of singularities, see, e.g., SGA 7 I 1.3),
there is an integer $m_0$ such that all eigenvalues of $M_x$ on the cohomology groups of
$F_x$ have order dividing $m_0$.
Thus
$
 \sum_{m \geq 1}  \Lambda (M_x^m) T^m
$
can be rewritten as
$\sum_{1 \leq i \leq m_0} \Lambda (M_x^i) \,  \frac{T^i}{ 1 - T^{m_0}} $
and the equality 
$ \chi_c (\mathcal{S}_{f, x}) = 
\Lambda (M_x^{m_0}) =  \chi_c (F_x)
$ follows.
\end{rema}

\bibliographystyle{amsplain}

\begin{thebibliography}{SGA}






\bibitem{ACampo1}N. A'Campo,
\textit{ Le nombre de Lefschetz d'une monodromie},  Indag. Math., \textbf{35} (1973), 113--118.

\bibitem{ACampo2}N. A'Campo,
\textit{ La fonction z\^eta d'une monodromie}, Comment. Math. Helv., \textbf{50}, (1975) 233--248.

\bibitem{berket}V. Berkovich, \textit{\'{E}tale cohomology for non-Archimedean analytic spaces},
 Publ. Math., Inst. Hautes \'{E}tud. Sci., \textbf{78} (1993), 5--171.




 \bibitem{berkv1}V. Berkovich, \textit{Vanishing cycles for formal schemes}, Invent. Math. \textbf{115} (1994), 539--571.


 \bibitem{berkv2}V. Berkovich, \textit{Vanishing cycles for formal schemes II}, Invent. Math. \textbf{125} (1996), 367--390.
 
 
\bibitem{berkv3}V. Berkovich, \textit{ Finiteness theorems for vanishing cycles of formal schemes}, preprint, available at \url{http://www.wisdom.weizmann.ac.il/~vova/}.

\bibitem{Brion}M. Brion,
\textit{Poly\`edres et r\'eseaux},  Enseign. Math., \textbf{38} (1992), 71--88.


\bibitem{DeligneLusztig}
P. Deligne, G. Lusztig,
\textit{Representations of reductive groups over finite fields},
Ann. of Math. \textbf{103} (1976), 103--161.



\bibitem{motivic}
J. Denef, F. Loeser, \textit{Motivic Igusa zeta functions}, J.
Algebraic Geom. \textbf{7} (1998), 505--537.


\bibitem{DLinv}J. Denef, F. Loeser,
\textit{Germs of arcs on singular algebraic varieties and motivic integration}, Invent. Math. \textbf{135} (1999), 201--232.

\bibitem{DLduke}J. Denef, F. Loeser,
\textit{Motivic exponential integrals and a motivic Thom-Sebastiani Theorem},
Duke Math. J. \textbf{99} (1999), 285--309.

\bibitem{barc}
J. Denef, F. Loeser, \textit{Geometry on arc spaces of algebraic
varieties}, Proceedings of 3rd European Congress of Mathematics,
Barcelona 2000, Progress in Mathematics \textbf{201} (2001),
327--348, Birkha{\"u}ser.


\bibitem{DLtop}J. Denef, F. Loeser,
\textit{Lefschetz numbers of iterates of the monodromy and truncated arcs},
Topology, \textbf{41} (2002), 1031--1040.



\bibitem{Dimca}A. Dimca, \textit{Singularities and topology of hypersurfaces},
Universitext. Springer-Verlag, New York, 1992.




 \bibitem{duc}A. Ducros, \textit{Parties semi-alg\'ebriques d'une vari\'et\'e alg\'ebrique $p$-adique},
Manuscripta Math. \textbf{111} (2003),  513--528.

\bibitem{fa}G. Faltings, \textit{The trace formula and Drinfeld's upper halfplane}, Duke Math. J. \textbf{76}
(1994), 467--481.

 \bibitem{HHK} D. Haskell,
E. Hrushovski, D. Macpherson,
\textit{Definable sets in algebraically closed valued fields: elimination of imaginaries},
J. Reine Angew. Math. \textbf{597}
 (2006), 175--236.



\bibitem{metastable}
E. Hrushovski, \textit{Valued fields, metastable groups},
preprint, available at \url{http://math.huji.ac.il/~ehud/mst.pdf}.


\bibitem{HK}
E. Hrushovski, D. Kazhdan,
\textit{Integration in valued fields}, in
 Algebraic geometry and number theory, Progress in Mathematics 253, 261--405 (2006), Birkh\"auser.



\bibitem{HK2}
E. Hrushovski, D. Kazhdan,
\textit{ The value ring of geometric motivic integration, and the Iwahori Hecke algebra of $\mathrm{SL}_2$.
With an appendix by Nir Avni},
Geom. Funct. Anal. \textbf{17} (2008),  1924--1967.  



\bibitem{HK3}
E. Hrushovski, D. Kazhdan,  
\textit{Motivic Poisson summation},  Mosc. Math. J.  \textbf{9}  (2009),   569--623.
  
  
\bibitem{Ishida}M. Ishida,
\textit{Polyhedral Laurent series and Brion's equalities}, Internat. J. Math, \textbf{1}, (1990) 251--265.

  
\bibitem{KS}M. Kontsevich, Y.Soibelman,
\textit{Stability structures, motivic Donaldson-Thomas invariants and cluster transformations},
arXiv:0811.2435.




\bibitem{thuong}
Q.T. L\^e,
\textit{Proofs of the integral identity conjecture over algebraically closed fields},
arXiv:1206.5334.




\bibitem{Lipshitz}
L. Lipshitz, \textit{Rigid subanalytic sets},
Amer. J. Math. \textbf{115} (1993), 77--108.


\bibitem{seattle}F. Loeser,
\textit{Seattle lectures on motivic integration}, in Algebraic geometry--Seattle 2005. Part 2, 745--784,
Proc. Sympos. Pure Math., 80, Part 2, Amer. Math. Soc., Providence, RI, 2009.


\bibitem{LS}F. Loeser, J. Sebag,
\textit{Motivic integration on smooth rigid varieties and invariants of degenerations},
Duke Math. J. \textbf{119} (2003),  315--344.


\bibitem{FM}F. Martin,
\textit{Cohomology of locally-closed semi-algebraic subsets}, arXiv:1210.4521.


\bibitem{Milnor}J. Milnor, \textit{Singular points of complex hypersurfaces}, Ann. Math. Stud., vol. 61. Princeton University Press, Princeton, NJ (1968).


\bibitem{Nicaise_ma}J. Nicaise,\textit{
A trace formula for rigid varieties, and motivic Weil generating series for formal schemes},
Math. Ann. \textbf{343} (2009), 285--349.


\bibitem{Nicaise_crelle}J. Nicaise,\textit{
A trace formula for varieties over a discretely valued field},
J. reine angew. Math. \textbf{650} (2011), 193--238.


\bibitem{NS}J. Nicaise, J. Sebag, \textit{
Motivic Serre invariants, ramification, and the analytic Milnor fiber},
Invent. Math. \textbf{168} (2007), 133--173.







\bibitem{vddries-tame}L. van den Dries, \textit{
Tame topology and o-minimal structures}, Cambridge Univ.
Press, New York, 1998.


\end{thebibliography}

\end{document}